\documentclass[mathpazo]{csam1}

\usepackage{tabularx}
\usepackage{color} 
\usepackage{diagbox}
\usepackage{graphics}
\usepackage{multirow}
\usepackage{bm}
\usepackage{float} 
\usepackage{amsmath}
\usepackage{epsfig}
\usepackage{amssymb,amsfonts}
\usepackage{epstopdf}
\usepackage{tikz}
\usepackage{pgflibraryarrows}
\usepackage{pgflibrarysnakes}

\newcommand{\R}{\mathbb{R}}

\newcommand{\vW}{\bm{W}}
\newcommand{\vw}{\bm{w}}
\newcommand{\vtheta}{\bm{\theta}}
\newcommand{\vx}{\bm{x}}
\newcommand{\mW}{\bm{W}}
\newcommand{\vb}{\bm{b}}
\newcommand{\va}{\bm{a}}
\newcommand{\vc}{\bm{c}}
\newcommand{\vxi}{\bm{\xi}}
\newcommand{\sR}{\mathbb{R}}
\newtheorem{thm}{Theorem}
\newcommand{\T}{\intercal}
\newcommand{\abs}[1]{\lvert#1\rvert} 
\newcommand{\D}{\mathrm{d}}
\newcommand*\diff{\mathop{}\!\D}
\newcommand{\fF}{\mathcal{F}}
\newcommand{\E}{\mathrm{e}}
\newcommand{\I}{\mathrm{i}}
\newcommand{\norm}[1]{\lVert#1\rVert}
\newcommand{\TT}{\mathcal{T}}
\newcommand{\N}{\mathcal{N}}

\makeatletter
\newcommand*\bigcdot{\mathpalette\bigcdot@{.5}}
\newcommand*\bigcdot@[2]{\mathbin{\vcenter{\hbox{\scalebox{#2}{$\m@th#1\bullet$}}}}}
\makeatother

\begin{document}
%%%%% title : short title may not be used but TITLE is required.
% \title{TITLE}
% \title[short title]{TITLE}
\title{Implicit bias with Ritz-Galerkin method in understanding deep learning for solving PDEs}

%\titlerunning{Deep learning for solving PDEs}        % if too long for running head

 \author[J. Wang et~al.]{Jihong Wang\affil{1}\comma,
       Zhi-Qin John Xu\affil{2}\corrauth, Jiwei Zhang\affil{3}~and Yaoyu Zhang\affil{2}}
 \address{\affilnum{1}\ Beijing Computational Science Research Center, Beijing 100193, P.R. China. \\
           \affilnum{2}\ School of Mathematical Sciences, Institute of Natural Sciences, MOE-LSC and Qing Yuan Research Institute,
    Shanghai Jiao Tong University, Shanghai 200240, P.R. China.\\
    \affilnum{3}\ School of Mathematics and Statistics, and Hubei Key Laboratory of Computational Science, Wuhan University, Wuhan 430072, P.R. China.
    }
 \emails{{\tt jhwang@csrc.ac.cn} (J.~Wang), {\tt xuzhiqin@sjtu.edu.cn} (Z.~Xu),
          {\tt jiweizhang@whu.edu.cn} (J.~Zhang), {\tt zhyy.sjtu@sjtu.edu.cn} (Y.~Zhang)}

%%%%% author(s) :
% single author:
% \author[name in running head]{AUTHOR\corrauth}
% [name in running head] is NOT OPTIONAL, it is a MUST.
% Use \corrauth to indicate the corresponding author.
% Use \email to provide email address of author.
% \footnote and \thanks are not used in the heading section.
% Another acknowlegments/support of grants, state in Acknowledgments section
% \section*{Acknowledgments}
%\author[O.~Author]{Only Author\corrauth}
%\address{School of Mathematical Sciences, Beijing Normal University,
%Beijing 100875, P.R. China}
%\email{{\tt author@email} (O.~Author)}

% multiple authors:
% Note the use of \affil and \affilnum to link names and addresses.
% The author for correspondence is marked by \corrauth.
% use \emails to provide email addresses of authors
% e.g. below example has 3 authors, first author is also the corresponding
%      author, author 1 and 3 having the same address.
% \author[Z. Zhang et~al.]{Zhengru Zhang\affil{1}\comma\corrauth,
%       Author Chan\affil{2}~and Author Zhao\affil{1}}
% \address{\affilnum{1}\ School of Mathematical Sciences,
%          Beijing Normal University,
%          Beijing 100875, P.R. China. \\
%           \affilnum{2}\ Department of Mathematics,
%           Hong Kong Baptist University, Hong Kong SAR.}
% \emails{{\tt zhang@email} (Z.~Zhang), {\tt chan@email} (A.~Chan),
%          {\tt zhao@email} (A.~Zhao)}
% \footnote and \thanks are not used in the heading section.
% Another acknowlegments/support of grants, state in Acknowledgments section
% \section*{Acknowledgments}

%%%%% Begin Abstract %%%%%%%%%%%
\begin{abstract}
This paper aims at studying  the difference between Ritz-Galerkin (R-G) method and deep neural network (DNN) method in solving partial differential equations (PDEs) to better understand deep learning. To this end, we consider solving a particular Poisson problem, where the information of the right-hand side of the equation $f$ is only available at $n$ sample points, that is, $f$ is known at finite sample points. 
%while the bases (neuron) number is much larger than $n$, which is common in DNN-based methods. 
Through both theoretical and numerical studies, we show that solution of the R-G method converges to a piecewise linear function for the one dimensional (1D) problem or functions of lower regularity for high dimensional problems. With the same setting, DNNs however learn a relative smooth solution regardless of the dimension, this is, DNNs implicitly bias towards functions with more low-frequency components among all functions that can fit the equation at available data points. This bias is explained by the recent study of frequency principle (Xu et al., (2019) \cite{xu2019frequency} and Zhang et al., (2019) \cite{zhang2019explicitizing,luo2019on}). In addition to the similarity between the traditional numerical methods and DNNs in the approximation perspective, our work shows that the implicit bias in the learning process, which is different from traditional numerical methods, could help better understand the characteristics of DNNs.
\end{abstract}
%%%%% end %%%%%%%%%%%

%%%%% AMS/PACs/Keywords %%%%%%%%%%%
%\pac{}
\ams{35Q68, 65N30, 65N35%The information of the AMS subject classification can be found in http://mathscinet.ams.org/msc/msc2010.html
}

%%%% maketitle %%%%%
\maketitle

%%%% Start %%%%%%
\section{Introduction}
Deep neural networks (DNNs) become increasingly important in scientific computing fields
\cite{weinan2017deep,weinan2018deep,han2018solving,he2018relu,liao2019deep,siegel2019approximation,hamilton2019dnn,cai2019deep,wang2020mesh}.
A major potential advantage over traditional numerical methods is
that DNNs could overcome the curse of dimensionality in high-dimensional
problems. With traditional numerical methods, several studies have
made progress on the understanding of the algorithm characteristics
of DNNs. For example, by exploring
ReLU DNN representation of continuous piecewise linear function in FEM, the work \cite{he2018relu}
theoretically establishes that a ReLU DNN can accurately represent any linear finite element functions. In the aspect of the convergence behavior, the works \cite{xu2019training,xu2019frequency}
show a Frequency Principle (F-Principle) that DNNs often learn
low-frequency components first while most of the conventional methods (e.g.,
Jacobi method) exhibit the opposite convergence behavior---higher-frequency components are learned faster. These understandings could
lead to a better use of DNNs in practice, such as  DNN-based algorithms are proposed
based on the F-Principle to fast eliminate high-frequency error \cite{cai2020phase,liu2020multi}.

%The aim of this paper is to investigate the different behaviors of DNNs and
%traditional numerical method, Ritz-Galerkin method, in solving PDEs with
%different number of sample points. We denote $n$ as the sample number
%and $m$ as the basis number in the Ritz-Galerkin method or the neuron
%number in DNNs. In practice, the sample number cannot be too large. $n\approx m$ in the Ritz-Galerkin
%method, however, we often encounter $n\ll m$ in DNNs. It has attracted
%substantive attention that how DNNs can generalize well when $n\ll m$
%in both low- and high-dimensional problems \cite{zhang2016understanding,xu2019frequency}.
%In this paper, we show that the Ritz-Galerkin method does not generalize
%when $n\ll m$ empirically and theoretically. Then, we incorporate
%the F-Principle to show how DNN is different from the Ritz-Galerkin
%method when $n\ll m$. Our work indicates that with implicit bias of DNNs, the traditional methods, e.g., FEM \cite{he2018relu}, could 
%provides insights into understanding DNNs.

The aim of this paper is to investigate the different behaviors between DNNs and traditional numerical method, e.g., Ritz-Galerkin (R-G) method. To this end, we utilize an example to show their stark difference, that is, solving PDEs given a few sample points. We denote $n$ by the sample number and $m$ by the basis number in the Ritz-Galerkin method or the neuron number in DNNs. In traditional PDE models, we consider the situation where the source functions in the equation are completely known, i.e. the sample number $n$ can go to infinity. %And for R-G method, we have the   
But in practical applications, such as signal processing, statistical mechanics, chemical and biophysical dynamic systems, we often encounter the problems that only a few sample values can be obtained. It is interesting to ask what effect R-G methods would have on solving this particular problem, and what the solution would be obtained by the DNN method. 
%It has attracted substantive attention that DNNs can generalize well when $n\ll m$ in both low- and high-dimensional problems \cite{zhang2016understanding,xu2019frequency}. 

In this paper, we show that R-G method considers the discrete sampling points as linear combinations of Dirac delta functions, while DNN methods always use a relatively smooth function to interpolate the discrete sampling points.  And we incorporate the F-Principle to show how DNN is different from the R-G method, that is, for all functions that can fit the training data, DNNs implicitly bias towards functions with more low-frequency components. In addition to the similarity between the traditional numerical methods and DNNs in the approximation perspective \cite{he2018relu}, Our work shows that the implicit bias in the learning process, which is different from traditional numerical methods, could  help better understand the characteristics of DNNs.

The rest of the paper is organized as follows. In section 2, we briefly introduce the R-G method and the DNN method. In sections 3 and 4, we present the difference between using these two methods to solve PDEs theoretically and numerically. We end the paper with the conclusion in section 5.
\section{Preliminary}
In this section we take the toy model of Poisson's equation as an example to investigate the difference of solution behaviors between R-G method and DNN method.

\subsection{Poisson problem}
We consider the $d$-dimensional Poisson problem posed on the bounded domain $\Omega \subset \mathbb{R}^d$ with Dirichlet boundary condition as 
\begin{eqnarray}\label{Model}
\left\{\begin{array}{c}
 - \Delta u(\bm{x})=f(\bm{x}),\quad \bm{x}\in\Omega,\\
 u(\bm{x})=0, \quad \bm{x} \in \partial\Omega,
\end{array}
\right.
\end{eqnarray}
where $\displaystyle \Delta$ represents the Laplace operator, $\bm{x}=(x_1,x_2,\cdots, x_d)$ is a d-dimensional vector. It is known that the problem (\ref{Model})
admits a unique solution for $f\in L^2(\Omega)$, and its regularity can be raised to $C_b^{s+2}(\Omega)$ if $f\in C_b^s(\Omega)$
for some $s\ge 0$.
 In literatures, there are a number of effective numerical methods to solve problem \eqref{Model} in general case. We here consider a special situation: we only have the information of $f(\bm{x})$ at the $n$ sample points $\bm{x}_i$ $(i =1,\cdots,n)$. In practical applications, we may imagine that we only have finite experiment data, i.e., the value of  $f(\bm{x}_i)$ $(i =1,\cdots,n)$, and have no more information of $f(\bm{x})$ at other points. 
%(ii) $f(\bm{x})$ is a given function and the dimension of spatial domain is very high. In this situation, the traditional methods based on grid points fails due to the curse of dimension, and the Monte Carlo (MC) method is feasible to simulate the problem \eqref{Model} by using deep learning method \cite{Goodfellow2016Deep}.
Through solving such a particular Poisson problem \eqref{Model} with R-G method and deep learning method, we aim to find the bias of these two methods in solving PDEs.

\subsection{R-G method}
In this subsection, we briefly introduce the R-G method \cite{Brenner2008The}. 
For problem \eqref{Model}, we construct a functional
\begin{eqnarray}
J(u) = \frac12a(u,u)-(f,u),
\end{eqnarray}
where
$$
a(u,v) = \int_\Omega \nabla u(\bm{x})\nabla v(\bm{x}) d\bm{x} ,\quad (f,v)=\int_\Omega f(\bm{x})v(\bm{x}) d\bm{x}. 
$$
The variational form of problem \eqref{Model} is the following:
\begin{eqnarray}\label{MPE}
\text{Find}\; u\in H_0^1(\Omega),\; \text{s.t.} \; J(u)=\min_{v\in H_0^1(\Omega)} J(v).
\end{eqnarray}
%the principle of minimal potential energy told us that if $u\in C^2(\Omega)\cap H_0^1(\Omega)$ is the minimum of $J(u)$, then $u$ is the solution of problem \eqref{Model}.
The weak form of \eqref{MPE} is to find $u\in H_0^1(\Omega)$ such that  
\begin{eqnarray}\label{VW}
a(u,v)=(f,v), \quad \forall \;v \in H_0^1(\Omega).
\end{eqnarray}
The problem (\ref{Model}) is the strong form if the solution $u\in H_0^2(\Omega)$. 
To numerically solve \eqref{VW}, we now introduce the finite dimensional space $U_h$ to approximate the infinite dimensional space $H_0^1(\Omega)$.  Let $U_h \subset H_0^1(\Omega)$ be a subspace with a sequence of basis functions $\{\phi_1,\phi_2,\cdots,\phi_m\}$. The numerical solution $u_h \in U_h$ that we will find can be represented as 
\begin{eqnarray}\label{ucphi}
u_h = \sum_{k=1}^{m} c_k\phi_k,
\end{eqnarray}
where the coefficients $\{c_i\}$ are the unknown values that we need to solve.
Replacing $H_0^1(\Omega)$ by $U_h$, both problems \eqref{MPE} and \eqref{VW} can be transformed to solve the following system:
%R-G method is that find the coefficients $c_i$, such that
\begin{eqnarray}\label{RG}
\sum_{k=1}^{m}c_ka(\phi_k,\phi_j)= (f,\phi_j),\quad j=1,2,\cdots,m.
\end{eqnarray}
From \eqref{RG}, we can calculate $c_i$, and then obtain the numerical solution $u_h$. We usually call \eqref{RG} R-G equation. 

For different types of basis functions, the R-G method can be divided into finite element method (FEM) and spectral method (SM) and so on. If the basis functions $\{\phi_i(\bm{x})\}$ are local, namely, they are compactly supported, this method is usually taken as the FEM. Assume that $\Omega$ is a polygon, and we divide it into finite element grid $\TT_h$ by simplex, $h=\max_{\tau\in\TT_h}\text{diam}(\tau)$.  A typical finite element basis is the linear hat basis function, satisfying
\begin{eqnarray}
\phi_k(\bm{x_j}) = \delta_{kj},\quad \bm{x_j}\in \N_h,
\end{eqnarray}
where $\N_h$ stands for the set of the nodes of grid $\TT_h$. The schematic diagram of the basis functions in 1-D and 2-D are shown in Fig. \ref{basis}.
%\begin{eqnarray}
%\phi_i(\bm{x}) = \left\{ \begin{array}{cc}
%a_i\cdot \bm{x}+b_i, & \bm{x} \in \Omega_i,\\
%0,  & \text{otherwise},
%\end{array}\right.
%\end{eqnarray}
On the other hand, if we choose the global basis function such as Fourier basis or Legendre basis \cite{Shen2011Spectral}, we call R-G method spectral method.

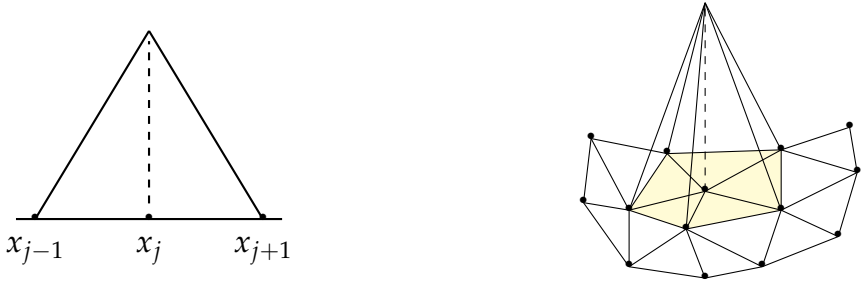
\begin{figure}[h!] 
\centering
\vspace{8pt}
\begin{minipage}[h]{0.45\linewidth}
\begin{tikzpicture}[font=\large,scale=0.5]
\centering
     \tikzset{to/.style={->,>=stealth',line width=.5pt}}
      \node(v1) at (0,0) {$\bigcdot$}; \node[yshift=-0.2cm]  at (v1.south) {$x_{j-1}$};   
      \node(v2) at (3,0) {$\bigcdot$}; \node[yshift=-0.2cm] at (v2.south) {$x_{j}$};  
      \node(v3) at (6,0) {$\bigcdot$}; \node[yshift=-0.2cm]  at (v3.south) {$x_{j+1}$};  
     \node(v4) at (3,5) {};  
      \draw[thick] (-0.5,0.01)--(6.5,0.01); \draw[thick] (0,0)--(3,5); \draw[thick] (6,0)--(3,5);
      \draw[thick, dashed] (v2)--(v4);
\end{tikzpicture}
\end{minipage}
\hspace{0.6cm}
\begin{minipage}[h]{0.45\linewidth}
\begin{tikzpicture}[font=\large,scale=0.5]
\centering
     \tikzset{to/.style={->,>=stealth',line width=.5pt}}
     \coordinate (v1) at (0,0); \coordinate (v7) at (0,5);
     \coordinate (v2) at (-2,-0.5); \coordinate (v3) at (-0.5,-1);
     \coordinate (v4) at (2.,-0.5); \coordinate (v5) at (2,1.1);
     \coordinate (v6) at (-1,1); 
     \draw[fill=yellow!20] (v2)--(v3)--(v4)--(v5)--(v6)--cycle; 
      \node at (v1) {$\bigcdot$};   \node at (v2) {$\bigcdot$}; 
      \node at (v3) {$\bigcdot$};   \node at (v4) {$\bigcdot$};
      \node at (v5) {$\bigcdot$};   \node at (v6) {$\bigcdot$}; 
      \draw (v1)--(v2);  \draw (v1)--(v3);  \draw (v1)--(v4);  \draw (v1)--(v5);  \draw (v1)--(v6); 
      \draw[dashed] (v1)--(v7); \draw (v2)--(v7); \draw (v3)--(v7); \draw (v4)--(v7); \draw (v5)--(v7);\draw (v6)--(v7); 
      \coordinate (e1) at (-3,1.4); \coordinate (e2) at (-3.2,-0.3); \coordinate (e3) at (-2,-2); \coordinate (e4) at (0,-2.3);     
      \coordinate (e5) at (1.5,-2); \coordinate (e6) at (3.5,-1.2); \coordinate (e7) at (4,0.5); \coordinate (e8) at (3.8,1.7);
       \node at (e1) {$\bigcdot$};   \node at (e2) {$\bigcdot$}; 
      \node at (e3) {$\bigcdot$};   \node at (e4) {$\bigcdot$};
      \node at (e5) {$\bigcdot$};   \node at (e6) {$\bigcdot$};   \node at (e7) {$\bigcdot$}; \node at (e8) {$\bigcdot$}; 
      \draw (e1)--(e2)--(e3)--(e4)--(e5)--(e6)--(e7)--(e8); 
      \draw (e1)--(v6);  \draw (e1)--(v2);  \draw (e2)--(v2);  \draw (e3)--(v2);  \draw (e3)--(v3);
      \draw (e4)--(v3);  \draw (e5)--(v3);  \draw (e5)--(v4);  \draw (e6)--(v4);  \draw (e7)--(v4); \draw (e7)--(v5);   
      \draw (e8)--(v5); 
\end{tikzpicture}
\end{minipage}
      \caption{The finite element basis function in 1d and 2d.}
      \label{basis}
\vspace{8pt}
 \end{figure}

The error estimate theory of R-G method has been well established. Under suitable assumption on the regularity of solution,  the linear finite element solution $u_h$ has the following error estimate 
$$\Vert u-u_h \Vert_{1} \leq C_1h|u|_{2},$$
where the constant $C_1$ is independent of grid size $h$. The spectral method has the following error estimate 
$$\Vert u-u_h \Vert \leq \frac {C_2}{m^s},$$
where $C_2$ is a constant and the exponent $s$ depends only on the regularity (smoothness) of the solution $u$. If $u$ is smooth enough and satisfies certain boundary conditions, the spectral method has the spectral accuracy. 

In this paper, we use the R-G method to solve the Poisson problem \eqref{Model} in a special setting, i.e. we only have the information of $f(\bm{x})$ at the $n$ sample points $\bm{x}_i$ $(i =1,\cdots,n)$. In this situation, the integral on the right hand side (r.h.s.) of R-G equation \eqref{RG} is hard to be computed exactly, so we need to compute it with the proper numerical method. For higher dimensional case, it is known the Monte Carlo (MC) integration \cite{Christian2004Monte} may be the only viable approach. Then replacing the integral on the r.h.s. of \eqref{RG} with the form of MC integral formula, we obtain
%In the one-dimensional case, \rd{if we know the properties of $f(\bm{x})$, we might be able to compute this integral with high order precision based on these $n$ points. }
%For higher dimensions, it is known the Monte Carlo (MC) method \cite{Christian2004Monte} is one of the best tools to compute the integral on the r.h.s. of equation \eqref{RG}. 
%This is, we have the value of  $f(\bm{x}_i)$ with $i =1,\cdots,n$, and no more information of $f(\bm{x})$ at other points.  This idea comes from the Monte Carlo (MC) method \citep{Christian2004Monte} to numerically solve the PDEs in high-dimensional cases.  If the dimension is very high, the traditional grid methods will be inefficient and have the curse of dimensionality. In this situation, MC method is a successful approach to simulate the high-dimensional PDEs. We here use the idea of MC integration by taking the finite sample points, 
%Since DNN-based methods are often used in high-dimensional cases, it is good to use MC integral to approximate the integral on the r.h.s. of \eqref{RG}, and arrive at the following modified R-G equation}
\begin{eqnarray} \label{VarEqM}
\sum_{k=1}^{m}c_k a(\phi_k,\phi_j)= \frac{1}{n}\sum_{i=1}^nf(\vx_i)\phi_j(\bm{x}_i),\quad j=1,2,\cdots,m.
\end{eqnarray}
In fact, if we use the Gaussian quadrature rule to compute the integral on the r.h.s. of equation \eqref{RG}, we still have the similar form as \eqref{VarEqM}, except that $1/n$ is replaced by the  corresponding Gaussian quadrature weights $w_i$. Because of the inaccuracy of the right-hand side integral, equation \eqref{VarEqM} is actually different from the traditional R-G method. However, we can see clearly the bias of the R-G method under this special setting in the later numerical experiments.

%with the matrix form
%\begin{eqnarray} \label{DL}
%\begin{pmatrix}
%a(\phi_1,\phi_1)&\cdots&a(\phi_1,\phi_m) \\
%a(\phi_2,\phi_1)&\cdots&a(\phi_2,\phi_m)\\
%\vdots &\ddots&\vdots\\
%a(\phi_m,\phi_1)&\cdots&a(\phi_m,\phi_m)
%\end{pmatrix}
%\begin{pmatrix}
%c_1\\ c_2 \\ \vdots \\c_m
%\end{pmatrix} =
%\frac1n
%\begin{pmatrix}
%\phi_1(\bm{x}_1)&\cdots&\phi_1(\bm{x}_n) \\
%\phi_2(\bm{x}_1) &\cdots&\phi_2(\bm{x}_n)\\
%\vdots &\ddots&\vdots\\
%\phi_m(\bm{x}_1)&\cdots&\phi_m(\bm{x}_n)
%\end{pmatrix}
%\begin{pmatrix}
%f_1\\ f_2 \\ \vdots \\f_n
%\end{pmatrix}.
%\end{eqnarray}

\subsection{DNN method}
We now introduce the DNN method. The $L$-layer neural network is denoted by
\begin{equation}
u_{\vtheta}(\vx)= \vW^{[L-1]} \sigma\circ(\mW^{[L-2]}\sigma\circ(\cdots (\mW^{[0]} \vx + \vb^{[0]} )\cdots)+\vb^{[L-2]})+\vb^{[L-1]},\label{generalDnn}
\end{equation}
where $\mW^{[l]} \in \sR^{m_{l+1}\times m_{l}}$, $\vb^{[l]}=\sR^{m_{l+1}}$, $m_0=d$, $m_{L}=1$,
$\sigma$ is a scalar function and ``$\circ$'' means entry-wise operation. 
We denote the set of parameters by \[
\vtheta=(\mW^{[0]},\mW^{[1]},\ldots,\vW^{[L-1]},\vb^{[0]},\vb^{[1]},\ldots,\vb^{[L-1]}),
\] 
and  an entry of $\vW^{[l]}$ by   $\vW^{[l]}_{ij}$.

%If the activation function in a one-hidden layer DNN is selected as the form of $\phi (\bm{x})$ in \eqref{ucphi},
Particularly, the one-hidden layer DNN with activation function $\sigma$ is given as 
\begin{eqnarray}
u_{\theta}(\bm{x}) = \sum_{k=1}^{m} c_k{\sigma}({w}_k \cdot \bm{x}+ b_k),
\end{eqnarray}
where $w_k\in \R^d, c_k, b_k\in \R$ are parameters. If we denote $\sigma({w}_k \cdot \bm{x}+ b_k)=\phi_k(\vx)$, then we obtain the similar form to R-G solution \eqref{ucphi}, i.e.,
\begin{eqnarray}\label{relunn}
u_{\theta}(\bm{x}) =\sum_{k=1}^{m} c_k{\phi_k}(\vx).
\end{eqnarray}
The difference between the expressions of the solutions of these two methods is that the basis functions of the R-G solution are known, while the  bases of the DNN solution are unknown, and need to be obtained together with the coefficients through the gradient descent algorithm with a loss function.
%The basis functions in \eqref{ucphi} are given and we only need to acquire the coefficients $\{c_k\}$ by solving equation \eqref{RG}, while in the DNN method both the bases and the coefficients are unknown, which obtained through the gradient descent algorithm with a loss function. 
 
%The model in (\ref{relunn}) can be generalized to a normal DNN in (\ref{generalDnn}).

 The loss function corresponding to problem \eqref{Model} is given by 
%\begin{eqnarray}\label{dnnloss}
%L_{0}(u_{\theta},f) = \frac{1}{n} \sum_{i=1}^n(\Delta u_{\theta}(\bm{x}_i)+f(\bm{x}_i))^2+\beta \sum_{\bm{x}_i\in \partial\Omega}(u_{\theta}(\bm{x}_i))^2,\label{eq:lossdirect}
%\end{eqnarray}
\begin{eqnarray}\label{dnnloss}
L_{0}(u_{\theta},f) = \frac{1}{n} \sum_{i=1}^n\left( \Delta u_{\theta}(\bm{x}_i)+f(\bm{x}_i) \right)^2+ \beta \int_{\partial\Omega} u_{\theta}(\bm{x})^2 ds,\label{eq:lossdirect}
\end{eqnarray}
 or a variation form \cite{weinan2017deep}
%\begin{eqnarray}\label{dnnloss_variation}
%L_{1}(u_{\theta},f) =\frac{1}{n_1} \sum_{\bm{x}_i\in \Omega}\left(\frac{1}{2}|\nabla_{x}u_{\theta}(\bm{x}_i)|^{2}-f(\bm{x_i})u_{\theta}(\bm{x}_i)\right)+\beta \sum_{\bm{x}\in \partial\Omega}(u_{\theta}(\bm{x}_i))^2, \label{eq:Energy-1}
%\end{eqnarray}
\begin{eqnarray}\label{dnnloss_variation}
L_{1}(u_{\theta},f) = \frac{1}{n} \sum_{i=1}^n\left(\frac{1}{2}|\nabla_{x}u_{\theta}(\bm{x}_i)|^{2}-f(\bm{x}_i)u_{\theta}(\bm{x}_i)\right)+ \beta \int_{\partial\Omega} u_{\theta}(\bm{x})^2 ds, \label{eq:Energy-1}
\end{eqnarray}
where the last term is for the boundary condition and $\beta$ is a hyper-parameter.
% In numerical experiments, we use the loss function (\ref{eq:lossdirect}) to train DNNs. We remark that the result with the loss function  (\ref{eq:Energy-1}) is similar.

\subsection{Frequency Principle}
In this section, we illustrate and introduce a rigorous definition of the F-Principle.

We start from considering a two-layer neural network, following \cite{zhang2019explicitizing,luo2019on},
\begin{equation}
    f(\vx, \vtheta) = \sum_{j=1}^m a_j \sigma\left(\vw_j^\T\vx - \abs{\vw_j}c_j\right), \label{eq:relunn2}
\end{equation}
where $\vw_j, \vx \in \sR^d$, $\vtheta = (\va^\T,\vw_1^\T,\ldots,\vw_m^\T,\vc^\T)^\T$, $\va,\vc\in\sR^{m}$
and $\mW=(\vw_1,\ldots,\vw_m)^\T\in\sR^{m\times d}$, and 
$\sigma(z)=\max(z,0)$ ($z\in\sR$) is the activation function of ReLU. Note that this two-layer model is slightly different from the model in (\ref{relunn}) for easy calculation in \cite{zhang2019explicitizing,luo2019on}.      The target
function is denoted by $f(\vx)$. The network is trained by  mean-squared error (MSE) loss function 
\begin{equation}
  L=\int_{\sR^d}\frac{1}{2}\abs{f(\vx,\vtheta)-f(\vx)}^{2}\rho(\vx) \diff{\vx}, \label{eq:mseloss}
\end{equation}
where $\rho(\vx)$ is a probability density. 
 Considering finite samples, we have 
 \begin{equation}
     \rho(\vx)=\frac{1}{n}\sum_{i=1}^n\delta(\vx-\vx_i).
 \end{equation}
 For any function $g$ defined on $\sR^d$, we use the following convention of the Fourier transform and its inverse:
\begin{equation*}
    \textstyle{  
    \fF[g](\vxi)=\int_{\sR^d}g(\vx)\E^{-2\pi\I \vxi^\T\vx}\diff{\vx},\quad
    g(\vx)=\int_{\sR^d}\fF[g](\vxi)    \E^{2\pi\I \vx^\T\vxi}\diff{\vxi},
    }
\end{equation*}
where $\vxi\in\sR^{d}$ denotes the frequency. 

The study in \cite{zhang2019explicitizing,luo2019on} shows that when the neuron number $m$ is sufficient large, training the network in (\ref{eq:relunn2}) with gradient descent is described by the following differential equation
\begin{equation}
\partial _t \fF[h](\vxi) = -\Gamma(\vxi) \fF[(h-f)\rho](\vxi) \label{gdform}
\end{equation}
 with initial $\fF[h_{{\rm ini}}](\vxi)$. 
The long time solution of \eqref{gdform} is equivalent to solve the following optimization problem
\begin{align}
    &\min_{h-h_{\mathrm{ini}}\in F_{\gamma}}\int_{\sR^d}\Gamma^{-1}(\vxi)\abs{\fF[h-h_{{\rm ini}}](\vxi)}^{2}\diff{\vxi},\label{eq: minFPnorm}\\
    &{\rm s.t. } \quad  h(\vx_{i})=f(\vx_i)  \quad  {\rm for}  \quad  i=1,\cdots,n,  \label{sampleconstraint}  
\end{align}
where 
\begin{equation}
    \Gamma(\vxi)=\frac{\frac{1}{m}\sum_{j=1}^{m}\left(\norm{\vw_{j}(0)}^{2}+a_{j}(0)^{2}\right)}{\norm{\vxi}^{d+3}}+\frac{4\pi^{2}\frac{1}{m}\sum_{j=1}^{m}\left(\norm{\vw_{j}(0)}^{2}a_{i}(0)^{2}\right)}{\norm{\vxi}^{d+1}},
\end{equation}
here $\norm{\cdot}$ represents the $L^2$-norm, $w_{j}(0)$ and $a_{j}(0)$ represent initial parameters before training, and 
\begin{equation}
    F_{\gamma}=\{h|\int_{\sR^d}\Gamma^{-1}(\vxi)\abs{\fF[h](\vxi)}^{2}\diff{\vxi}<\infty\}.
\end{equation}
Since $\Gamma(\vxi)$ monotonically decreases with $\vxi$, the gradient flow in (\ref{gdform}) rigorously defines the F-Principle, i.e., low frequency converges faster. The minimization in (\ref{eq: minFPnorm}) clearly shows that the DNN has an implicit bias in addition to the sample constraint in (\ref{sampleconstraint}). As $(\Gamma(\vxi))^{-1}$ monotonically increases with $\vxi$, the optimization problem prefers to choose a function that has less high frequency components, which explicates the implicit bias of the F-Principle --- DNN prefers low frequency  \cite{xu2019training,xu2019frequency}.

\section{Main results}

\subsection{R-G method in solving PDE}

In the classical case, $f(\bm{x})$ is a given function, so we can compute exactly the integral on the right-hand side of the R-G equation \eqref{RG}. As the number of basis functions $m$ approaches infinity, the numerical solution obtained by R-G method \eqref{RG} approximates the exact solution of problem \eqref{Model}.  It is interesting to ask if we only have the information of $f$ at the finite $n$ points, what could happen to numerical solution obtained by \eqref{VarEqM} when $m \rightarrow \infty$?

Fixing the number of sample points $n$, we study the property of the solution of the numerical method \eqref{VarEqM}. We have the following theorem.
%And we try to prove that the solution of the problem \eqref{Model} obtained by R-G method as $m$ approaches infinity is the exact solution of the following continuous problem
%\begin{eqnarray}\label{Model1}
%-\Delta u(x) = \frac1M\sum_{i=1}^M \delta(x-x_i)f(x_i).
%\end{eqnarray}
%From this we can see how the R-G method deals with the discrete sampling points on the right end when solving this problem.
%
%Then we study the effect of deep learning method to solve the problem \eqref{Model}. We will find that this approach treats discrete values on the right as continuous functions.
%
%If we don't know the formula for the function $f$, but just know the value at some points, then we can't compute exactly the integral on the right side of the \eqref{FEM}. In the one-dimensional case, we still have some higher order numerical integration method to use, but in the higher dimensional case, Monte Carlo(MC) method may be the best as far as we know. Then there we use the MC method to compute the right integral. 

\begin{thm}\label{thm1}
When $m\rightarrow \infty$, the numerical method \eqref{VarEqM} is solving the problem
\begin{eqnarray}\label{Th1Eq}
\left\{\begin{array}{c}
\displaystyle -\Delta u(\bm{x}) = \frac{1}{n}\sum_{i=1}^n \delta(\bm{x}-\bm{x}_i)f(\bm{x}_i), \quad \bm{x}\in \Omega, \\
 u(\bm{x})=0, \quad \bm{x} \in \partial\Omega,
\end{array}
\right.
\end{eqnarray}
where $\delta(x)$ represents the Dirac delta function.
\end{thm}

\begin{proof}
According to the filtering property of delta function, the formula for the integral on the right hand side (r.h.s.) of the equation \eqref{VarEqM} is the exact integral between the function of the r.h.s. of equation \eqref{Th1Eq} and the basis function $\phi_j(x)$, i.e.,
$$ \frac{1}{n}\sum_{i=1}^n f(\bm{x}_i) \phi_j(\bm{x}_i)=\int_{\Omega}\frac{1}{n}\sum_{i=1}^n f(\bm{x}_i) \delta(\bm{x}-\bm{x}_i)\phi_j(\bm{x}) d\bm{x},\quad j=1,2,\cdots.$$
Note that $\delta(x): \R^d \rightarrow \R$ is a continuous linear functional. According to the error estimation theory \cite{Brenner2008The,Shen2011Spectral}, the finite dimensional system \eqref{VarEqM} evidently approximates the problem \eqref{Th1Eq} when $m\rightarrow \infty$.
\end{proof}

{\bf Remark:} %By integrating equation \eqref{Th1Eq} twice, we can assert that the numerical solution $u(x)$ is piecewise linear, since the $\delta$ function is a linear function when it is integrated twice. 
For the 1D case, the analytic solution to problem \eqref{Th1Eq} defined in $[a,b]$ can be given as  a piecewise linear function, namely 
\begin{eqnarray}\label{exactSol}
u(x)= \displaystyle \frac{1}{n}\sum_{i=1}^n f(x_i)(b-x_i)\frac{x-a}{b-a} -\frac{1}{n}\sum_{i=1}^n f(x_i)(x-x_i)H(x-x_i),
\end{eqnarray}
where $H(x)$ is the Heaviside step function
\begin{eqnarray*}
H(x)=\left\{ \begin{array}{cc}
0,&x<0,\\
1,&x\ge 0.
\end{array}
\right.
\end{eqnarray*} 
%Furthermore, we know from Theorem \ref{thm1} that the numerical solution of the problem \eqref{Model} obtained by the R-G method \eqref{VarEqM} is piecewise linear.  And we will verify this in later numerical experiments.
For the 2D case, \cite{Polyanin2002Handbook} gives the exact solution in $[0,a]\times[0,b]$ by Green's function
%Exact solution on $[0,1]^2$:
%$$u(x,y)=\frac{4}{n}\sum_{i=1}^nf_i\sum_{k=1}^{\infty}\sum_{l=1}^{\infty}\frac{\sin(\pi kx)\sin(\pi ly)\sin(\pi kx_i)\sin(\pi ly_i)}{\pi^2(k^2+l^2)}.$$
\begin{eqnarray}
u(x,y)=\frac{4}{nab}\sum_{i=1}^nf_i\sum_{k=1}^{\infty}\sum_{l=1}^{\infty}\frac{\sin(p_kx)\sin(q_ly)\sin(p_kx_i)\sin(q_ly_i)}{p_k^2+q_l^2}.
\end{eqnarray}
where $f_i=f(x_i,y_i), p_k=\pi k/a, q_l=\pi l/b$. We can prove that this series diverges at the sampling point $(x_i, y_i) (i=1,2,\cdots,n)$ and converges at other points. Therefore, the 2d exact solution $u(x,y)$ is highly singular.

\subsection{Numerical experiments}
Although R-G method and DNN method can be equivalent to each other in the sense of approximation, in this section, we present three examples to investigate the difference between the solution obtained by the R-G method and the one obtained by the DNN with gradient descent optimization. For simplicity, we consider the following 1D Poisson problem
\begin{eqnarray}\label{ExModel}
\left\{\begin{array}{c}
-u''(x)=f(x),\quad x\in(-1,1),\\
u(-1)=u(1)=0,
\end{array}
\right.
\end{eqnarray}
where we only know what the value of $f(x)$ is at $n$ points, i.e. $f(x_i)~(i=1,2,\cdots,n)$. In experiments, $f(x_i)$ are sampled from the function $f(x) =-(4x^3-6x)\exp(-x^2)$.
%where $f(x) =-(4x^3-6x)\exp(-x^2).$ 
%The problem \eqref{ExModel} has the exact solution 
%$$u(x) =  x(\exp(-x^2)-\exp(-1)).$$

\bigskip 
\noindent {\bf Example 1:}  Fixing the number of sampling points $n=5$, we use R-G method and DNN method to solve the problem \eqref{ExModel}, respectively. The reason why we choose fewer sample points here is that in this situation we can investigate the property of the solution more clearly. The results are shown as follows.

\noindent{\bf R-G  method.} 
First, we use R-G method to solve the problem \eqref{ExModel}, specially the spectral method with the Fourier basis function given as 
$$\phi_i(x) = \sin(k\pi x),\quad  k = 1,2,\cdots,m.$$
%$$\phi_i(x) = L_{i+1}(x)-L_{i-1}(x).$$
We set the numbers of basis functions $m=5,10, 50, 500$, respectively. Fig. \ref{RGM5} plots the numerical solutions and the exact solution \eqref{exactSol} of problem \eqref{Th1Eq} with different $m$. One can see that the R-G solution approximates the piecewise linear function \eqref{exactSol}  when $m \rightarrow \infty$. This result is consistent with the solution property analyzed in Theorem \ref{thm1}.
%
% \begin{figure}
%\centering
%\begin{minipage}[]{\textwidth}
%	\includegraphics[width=0.45\textwidth]{Fig1.eps} \hspace{0.2in} 
%	\includegraphics[width=0.45\textwidth]{Fig2.eps} \vspace{0.1in}\\
%	\includegraphics[width=0.45\textwidth]{Fig3.eps} \hspace{0.2in} 
%	\includegraphics[width=0.5\textwidth]{Fig4.eps}
%\end{minipage}
%\caption{Numerical solution in SM with fixed $n=10$}
%\label{RGM10}
%\end{figure}

 \begin{figure}
\centering
\begin{minipage}[]{\textwidth}
	\includegraphics[width=0.45\textwidth]{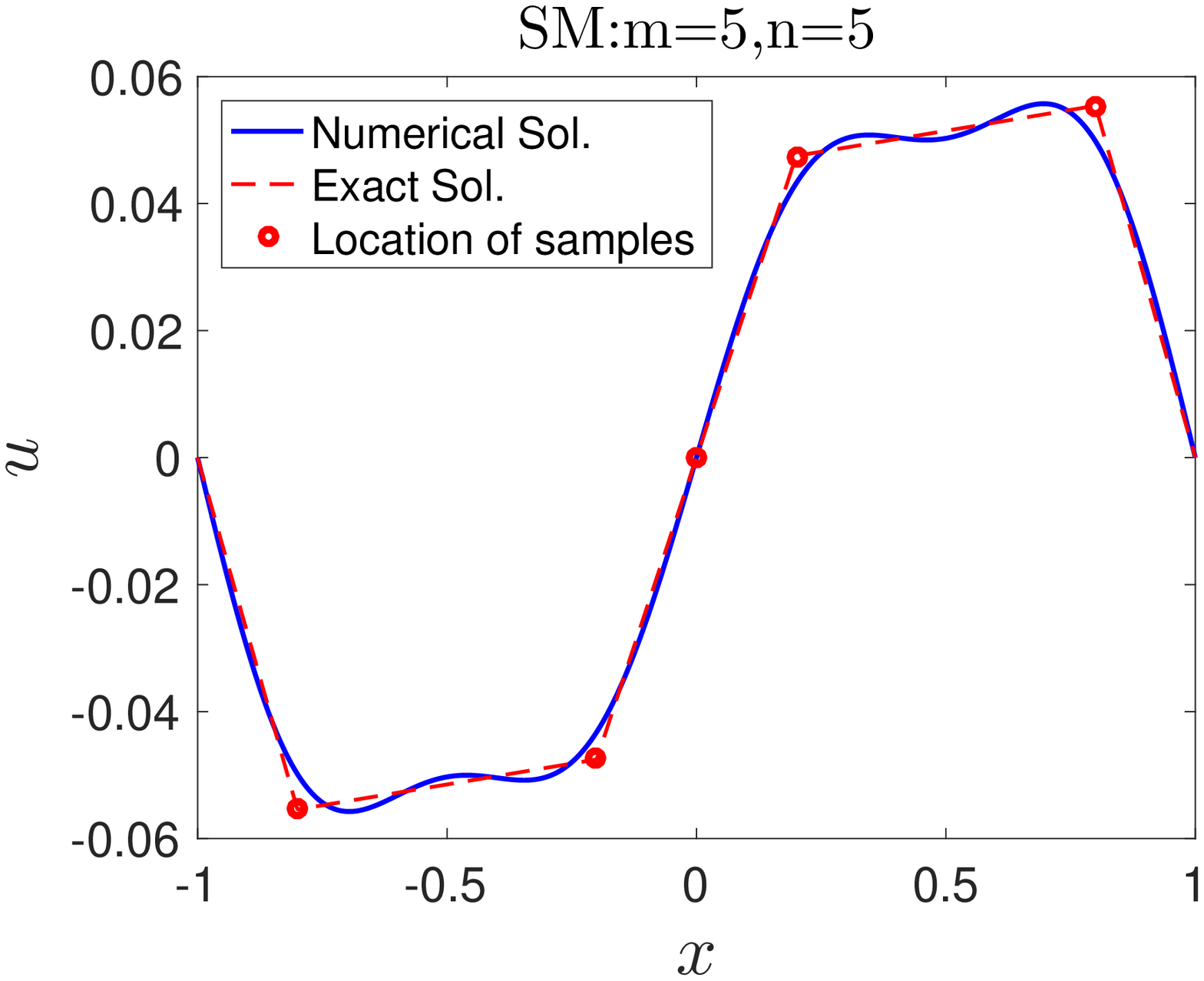} \hspace{0.2in} 
	\includegraphics[width=0.45\textwidth]{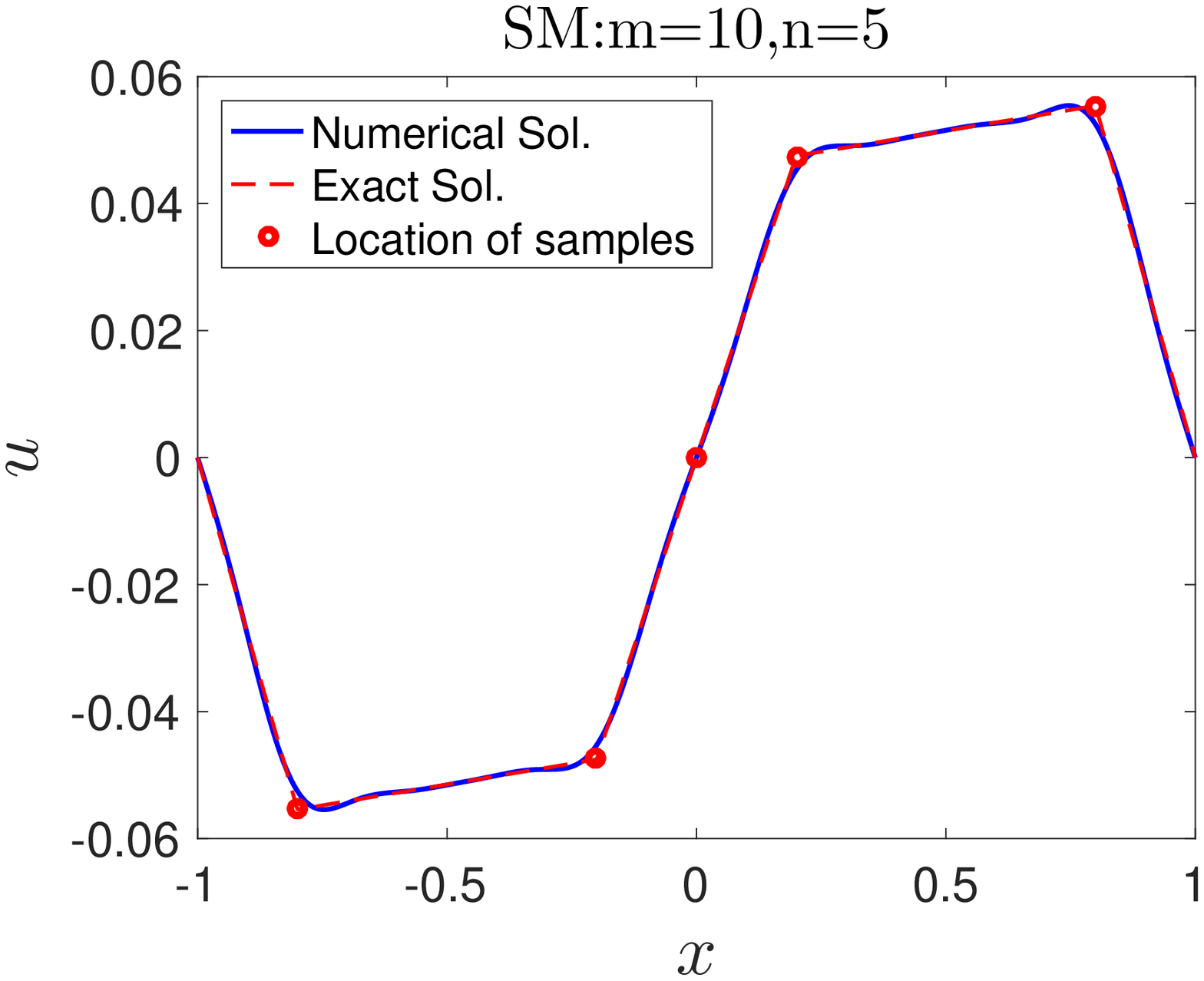} \vspace{0.1in}\\
	\includegraphics[width=0.45\textwidth]{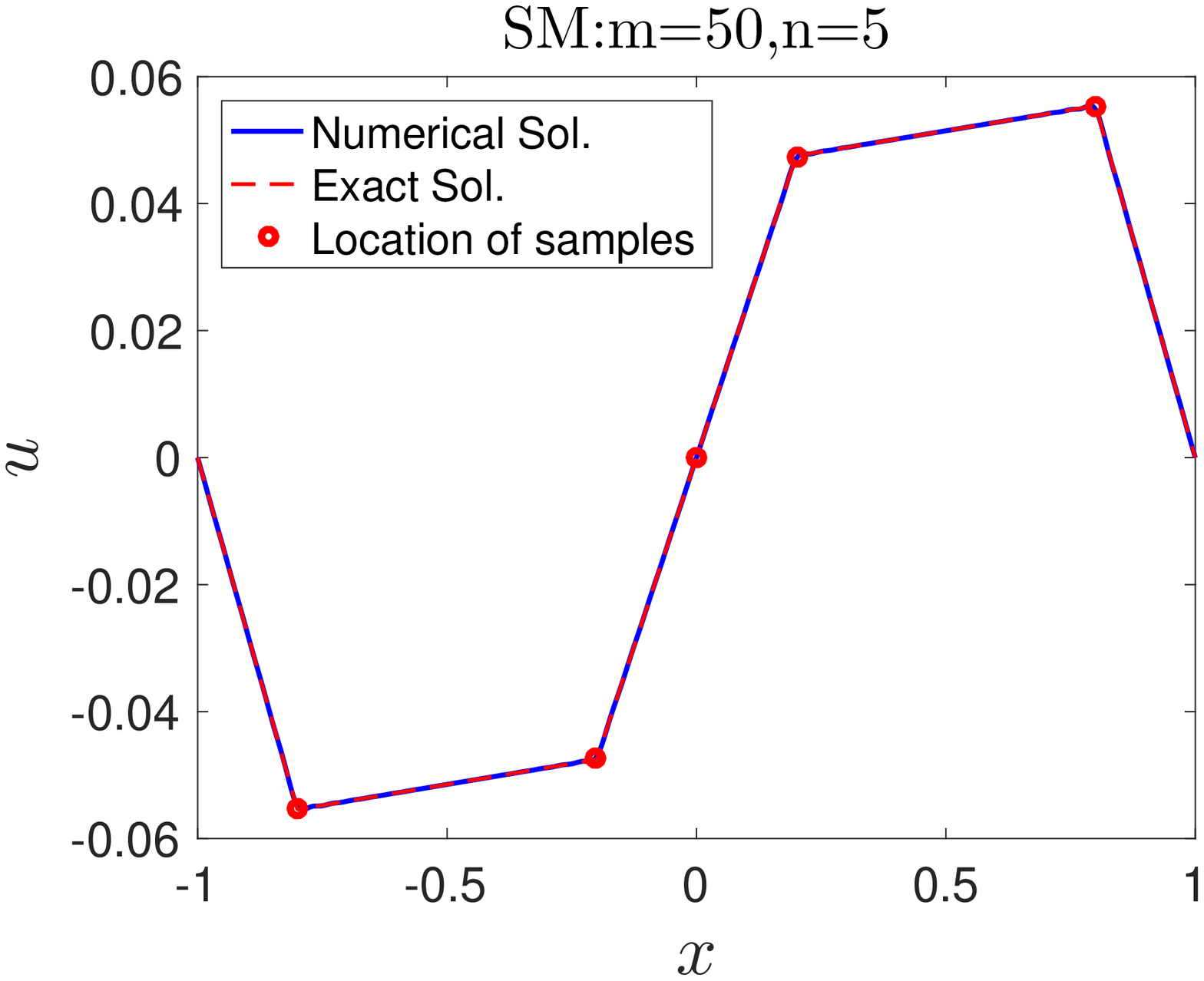} \hspace{0.2in} 
	\includegraphics[width=0.45\textwidth]{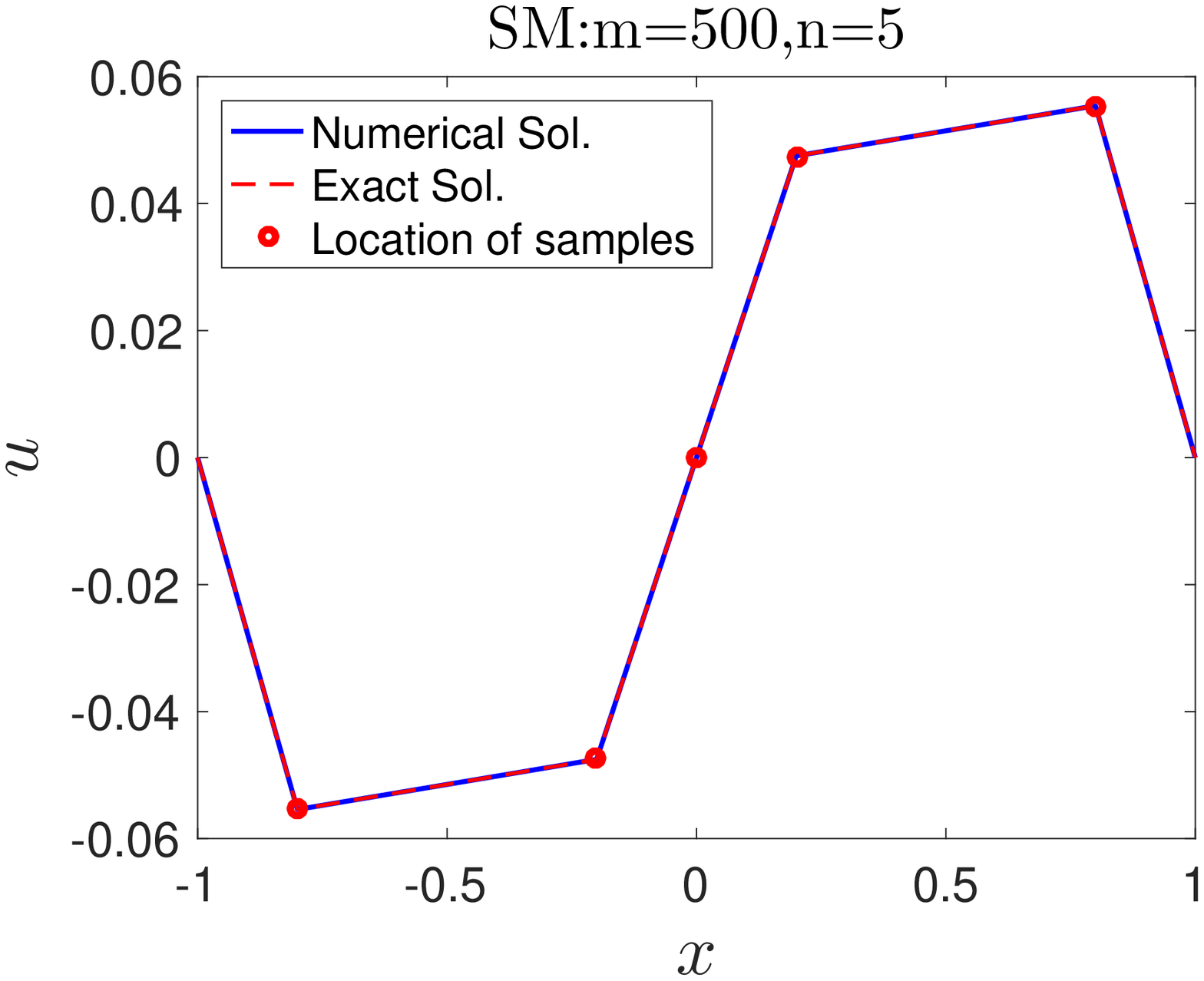}
\end{minipage}
\caption{(Example 1): Numerical solutions in SM with 5 sampling points.}
\label{RGM5}
\end{figure}

\noindent{\bf DNN method.}
For a better comparison with R-G method, we choose the activation function by $\sin(x)$ in DNN with one hidden layer. And the number of neurons are taken as $m=5,10, 50, 500$. The loss function \eqref{dnnloss} is selected with the parameter $\beta=10$.  We reduce the loss to an order of 1e-4, and take learning rate by 1e-4. And 1000 test points are used to plot the figures.
The DNN solutions are shown in Fig. \ref{DNNM5}, in which we observe that the DNN solutions are always smooth even when $m$ is very large.

 \begin{figure}
\centering
\begin{minipage}[]{\textwidth}
	\includegraphics[width=0.45\textwidth]{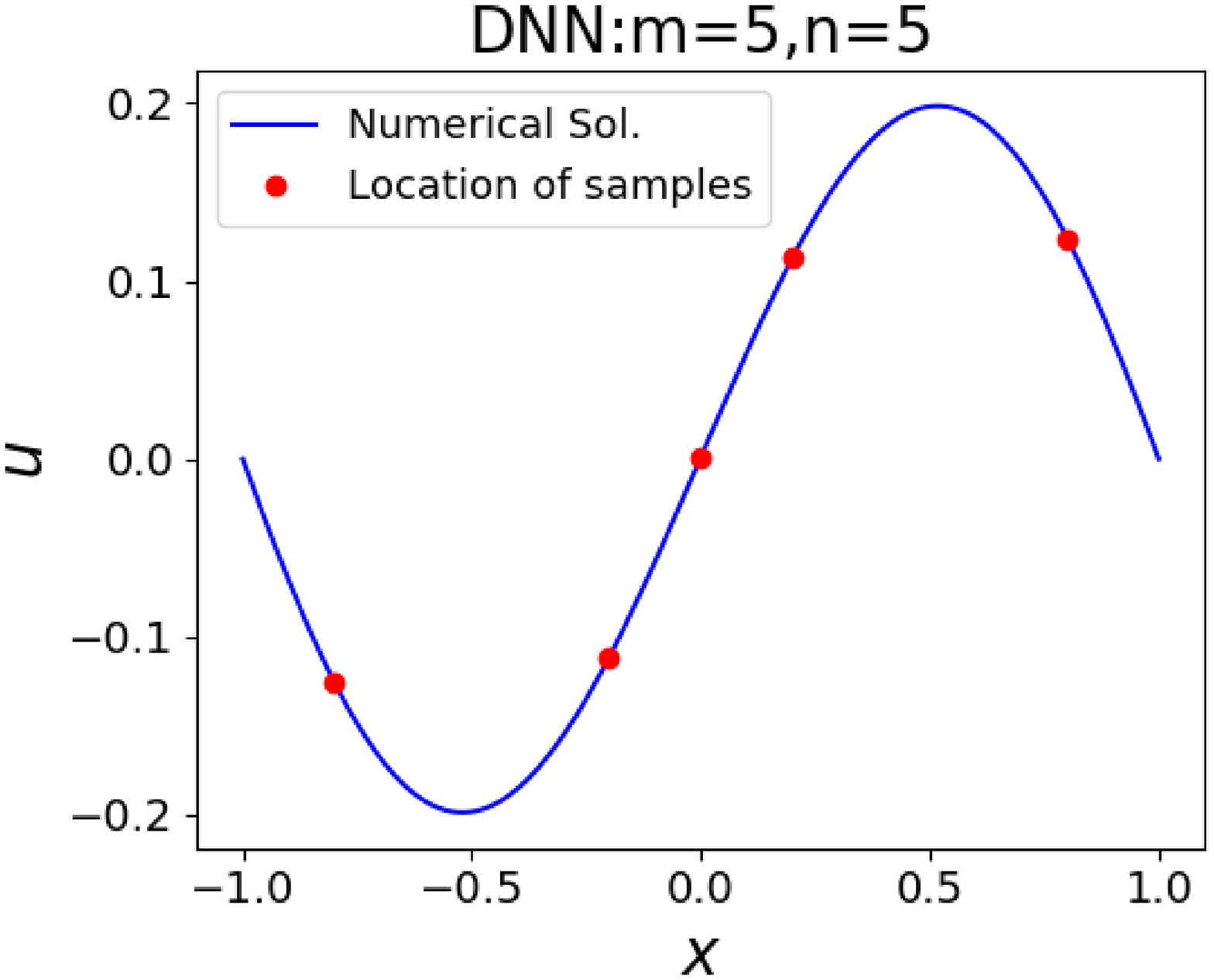} \vspace{0.1in}
	\includegraphics[width=0.45\textwidth]{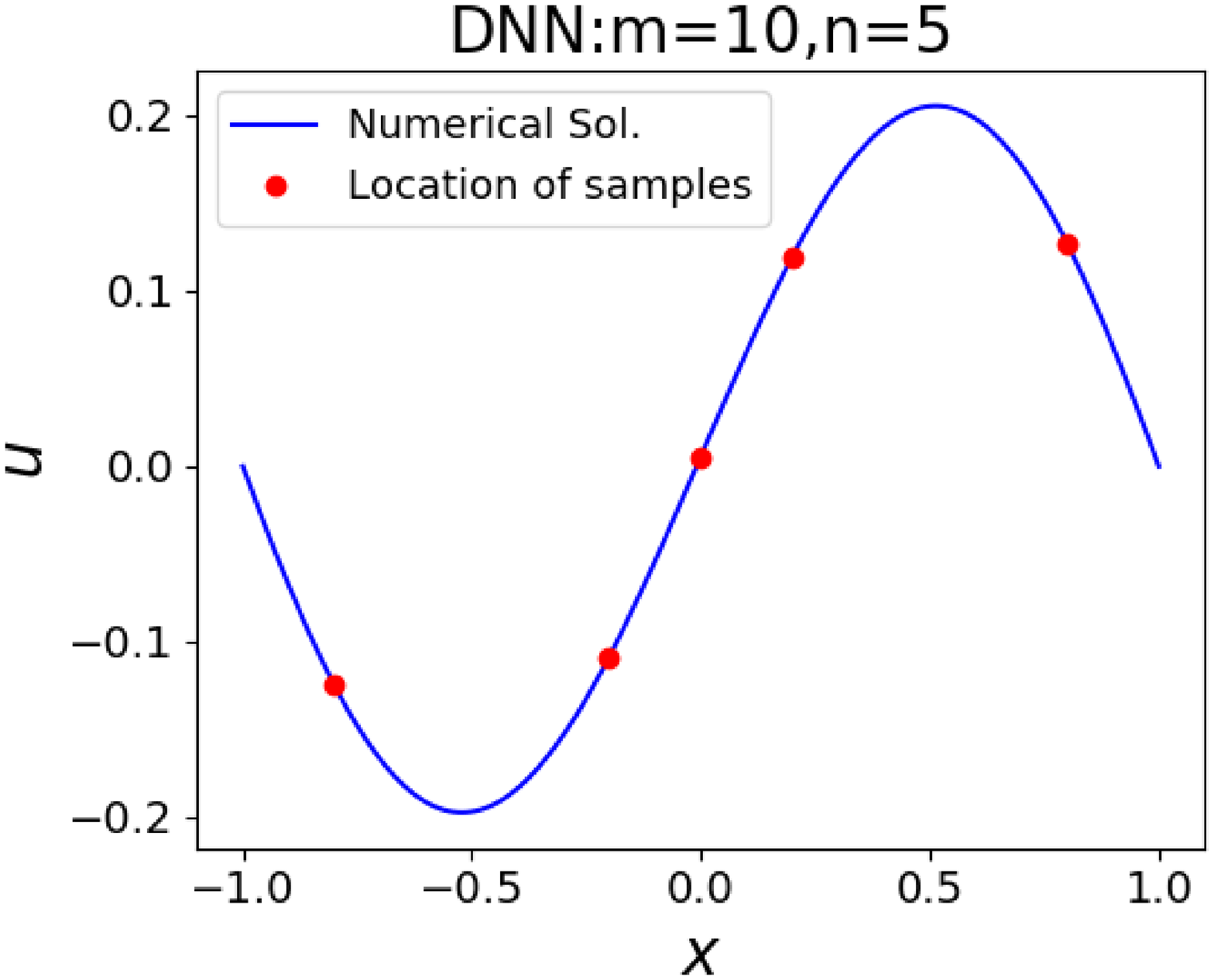} \vspace{0.1in}\\
	\includegraphics[width=0.45\textwidth]{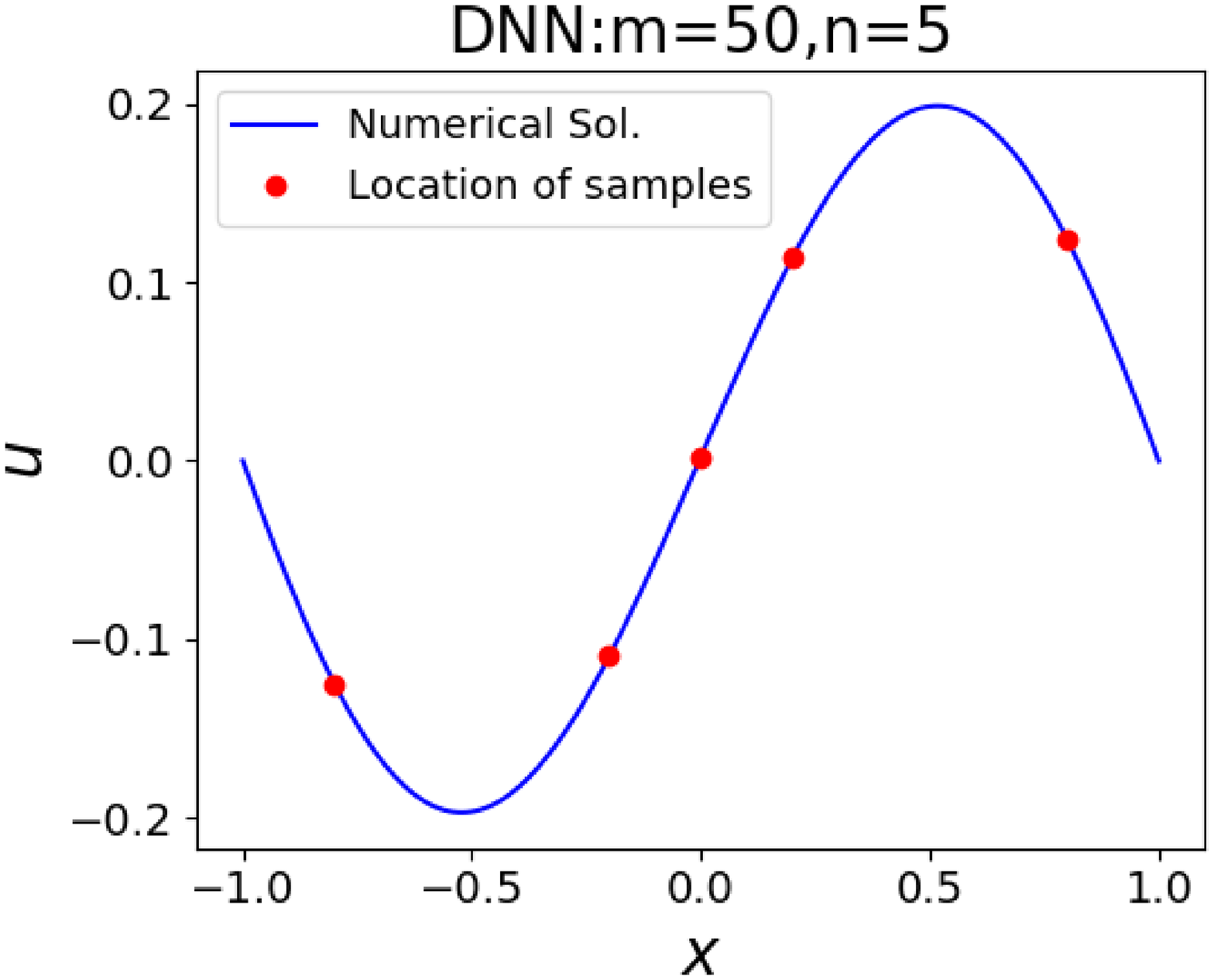} \vspace{0.1in}
	\includegraphics[width=0.45\textwidth]{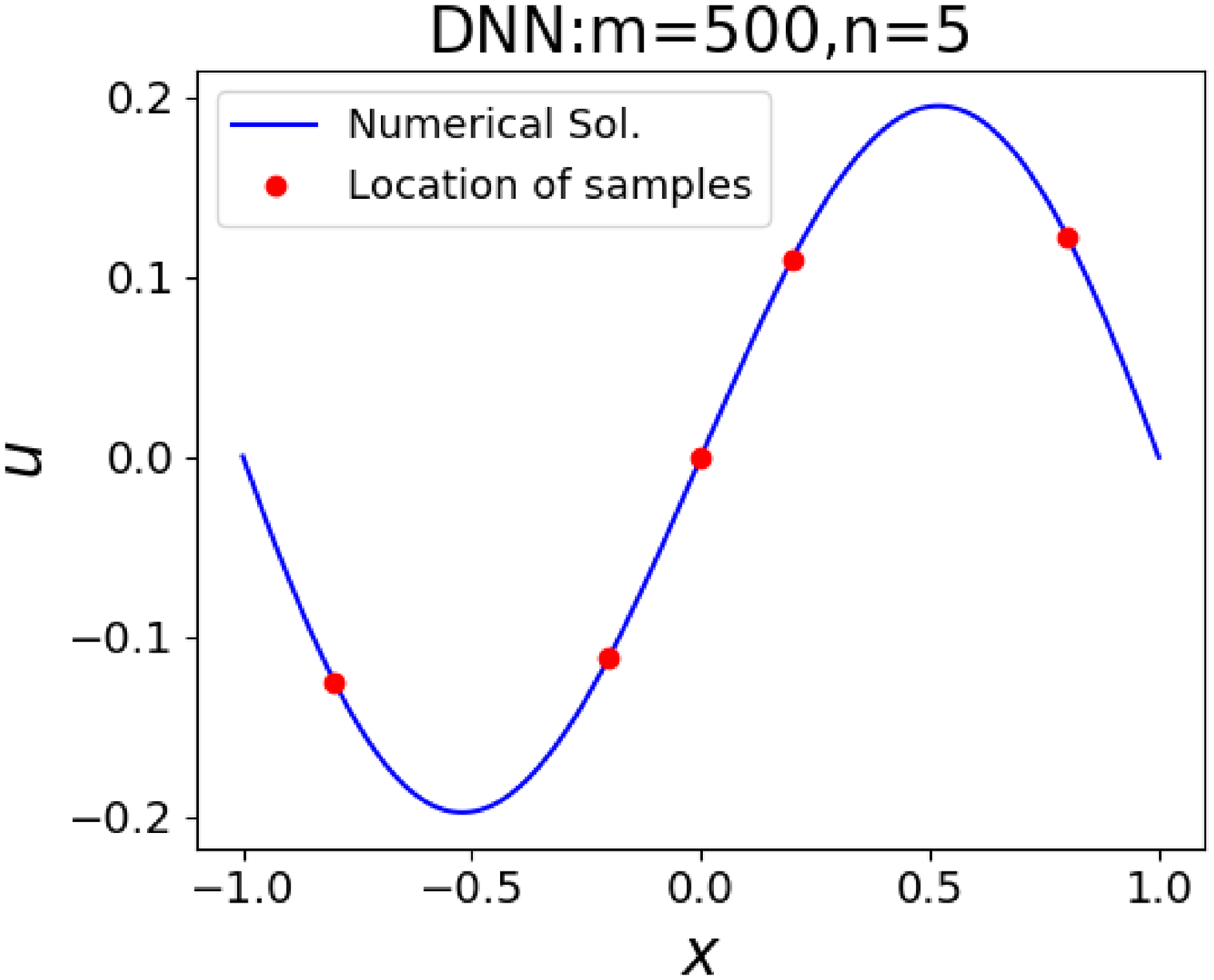} \vspace{0.1in}
\end{minipage}
\caption{(Example 1): Numerical solutions in DNN method with 5 sampling points.}
\label{DNNM5}
\end{figure}

\bigskip
\noindent{\bf Example 2:}
In this example, we use the ReLU function as the basis function in R-G method and the activation function in DNN method to repeat the experiments in Example 1. Here we randomly choose another 10 sampling points. 
%And the unstated parameters are taken as the same as those in Example 1.

Since the linear finite element function $\phi_j(x)$ can be expressed by ReLU functions for one dimensional case, namely, 
\begin{equation*}
\phi_j(x) = \frac{1}{h_{j-1}}\text{ReLU}(x-x_j)-(\frac{1}{h_{j-1}}+\frac{1}{h_j})\text{ReLU}(x-x_j)+\frac{1}{h_{j}}\text{ReLU}(x-x_{j+1}),
\end{equation*}
where $h_j=x_{j+1}-x_{j}$, we can use ReLU as the basis function for R-G method. For convenience, we just use the linear finite element function $\phi_j(x)$ instead of ReLU function as the basis function. Fig. \ref{FEMM5} shows that there is no doubt that the FEM solution approximates the piecewise linear solution \eqref{exactSol}.

In DNN method, we choose the number of neurons $m=5, 10,50,500$, respectively. And we use the variational from of the loss function \eqref{dnnloss_variation} because of the second order derivative of ReLU function is always zero. Fig. \ref{DNNrelu} shows that the DNN learns the data as a relatively smoother function than the R-G method.

\begin{figure}
\centering
\begin{minipage}[]{\textwidth}
	\includegraphics[width=0.45\textwidth]{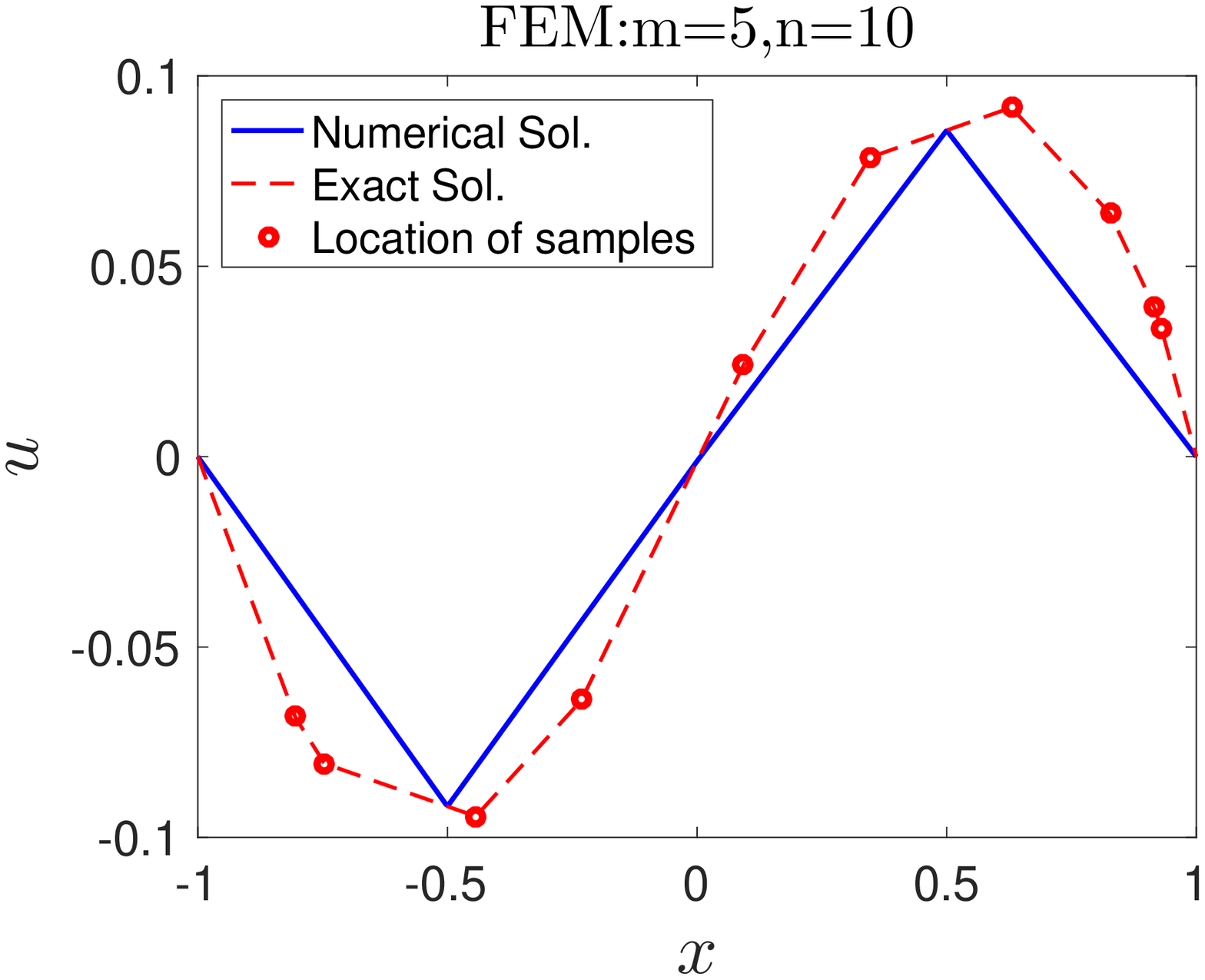} \hspace{0.2in} 
	\includegraphics[width=0.45\textwidth]{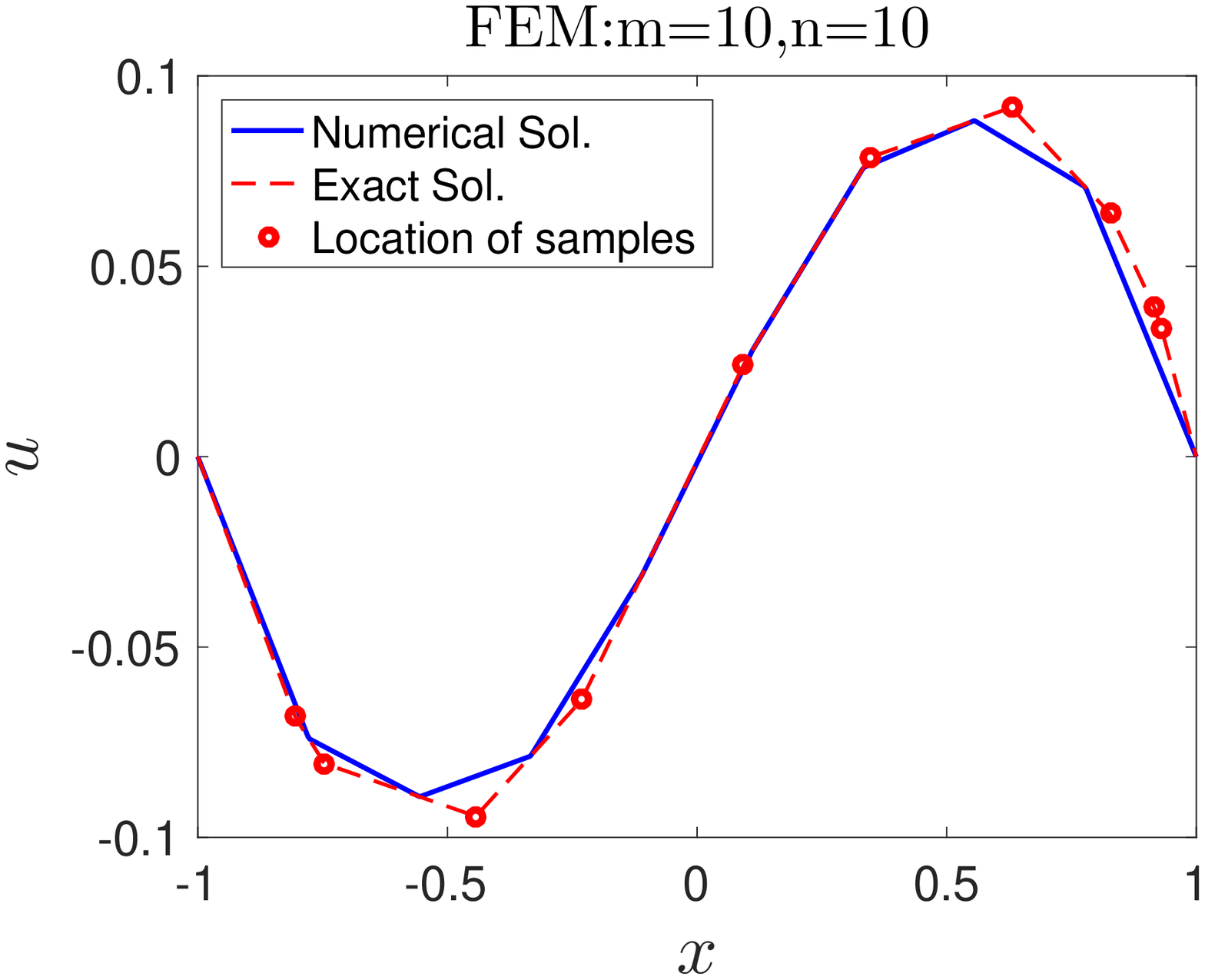} \vspace{0.1in}\\
	\includegraphics[width=0.45\textwidth]{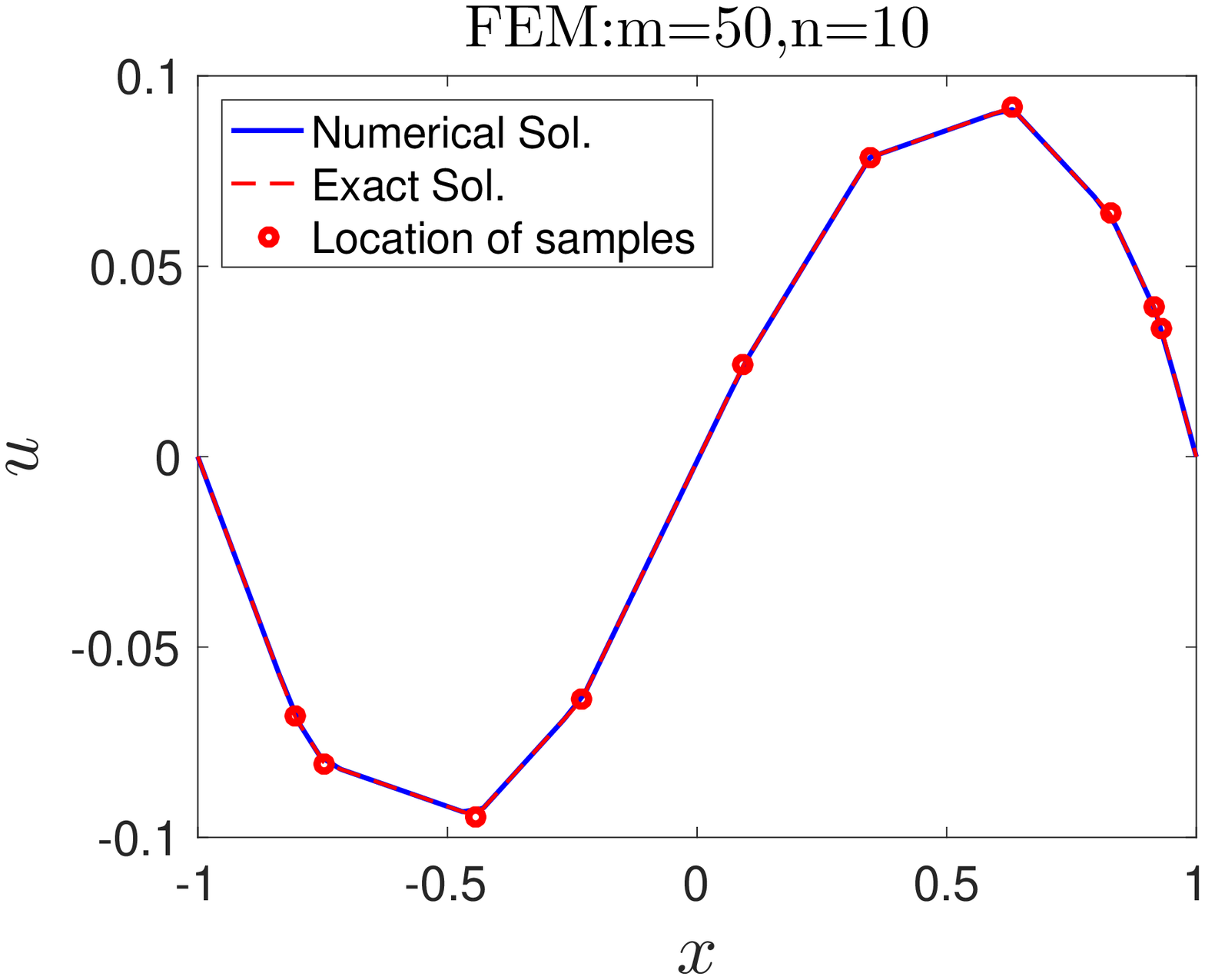} \hspace{0.2in} 
	\includegraphics[width=0.45\textwidth]{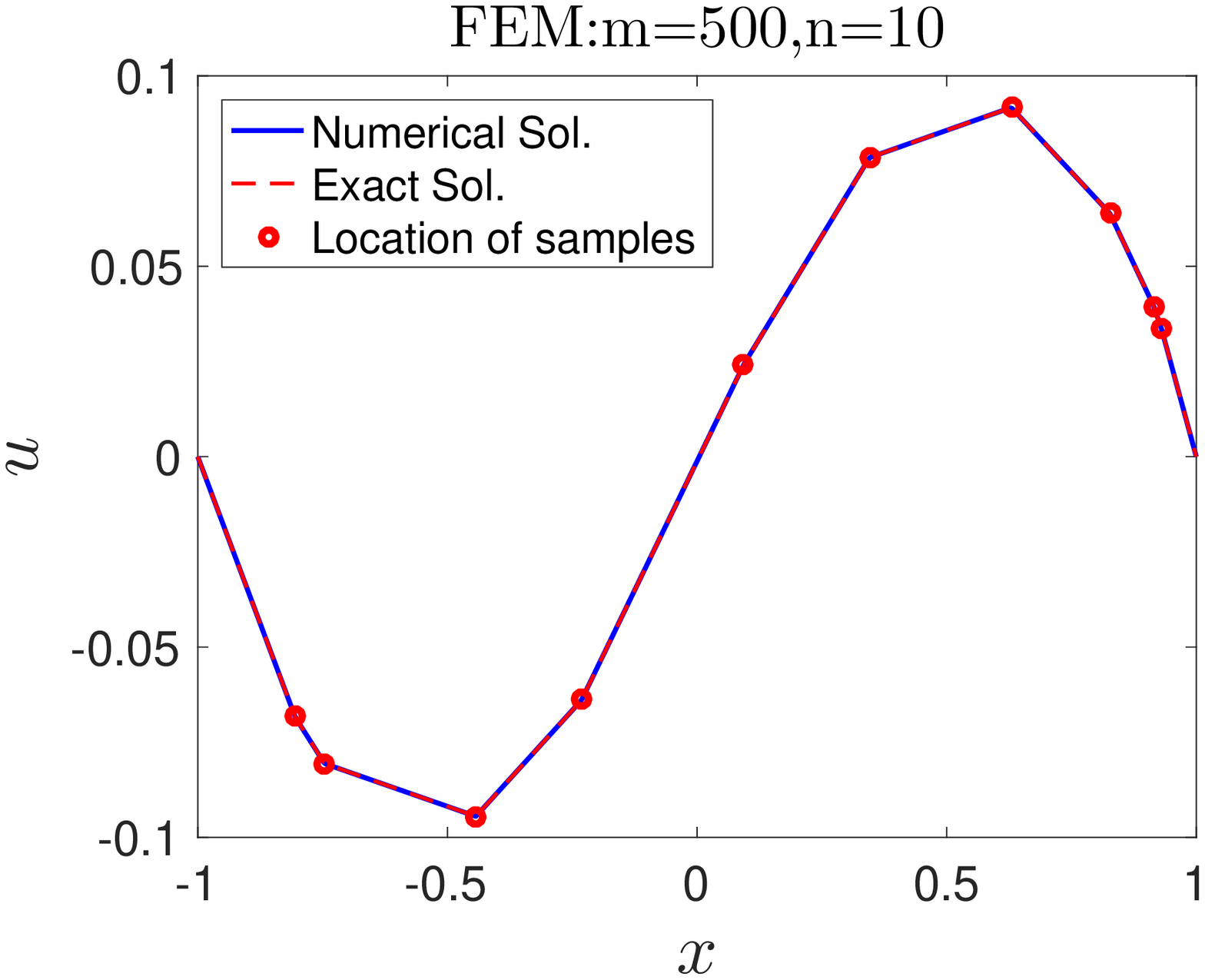}
\end{minipage}
\caption{(Example 2): Numerical solutions in FEM with 10 sampling points.}
\label{FEMM5}
\end{figure}

 \begin{figure}
\centering
\begin{minipage}[]{\textwidth}
	\includegraphics[width=0.45\textwidth]{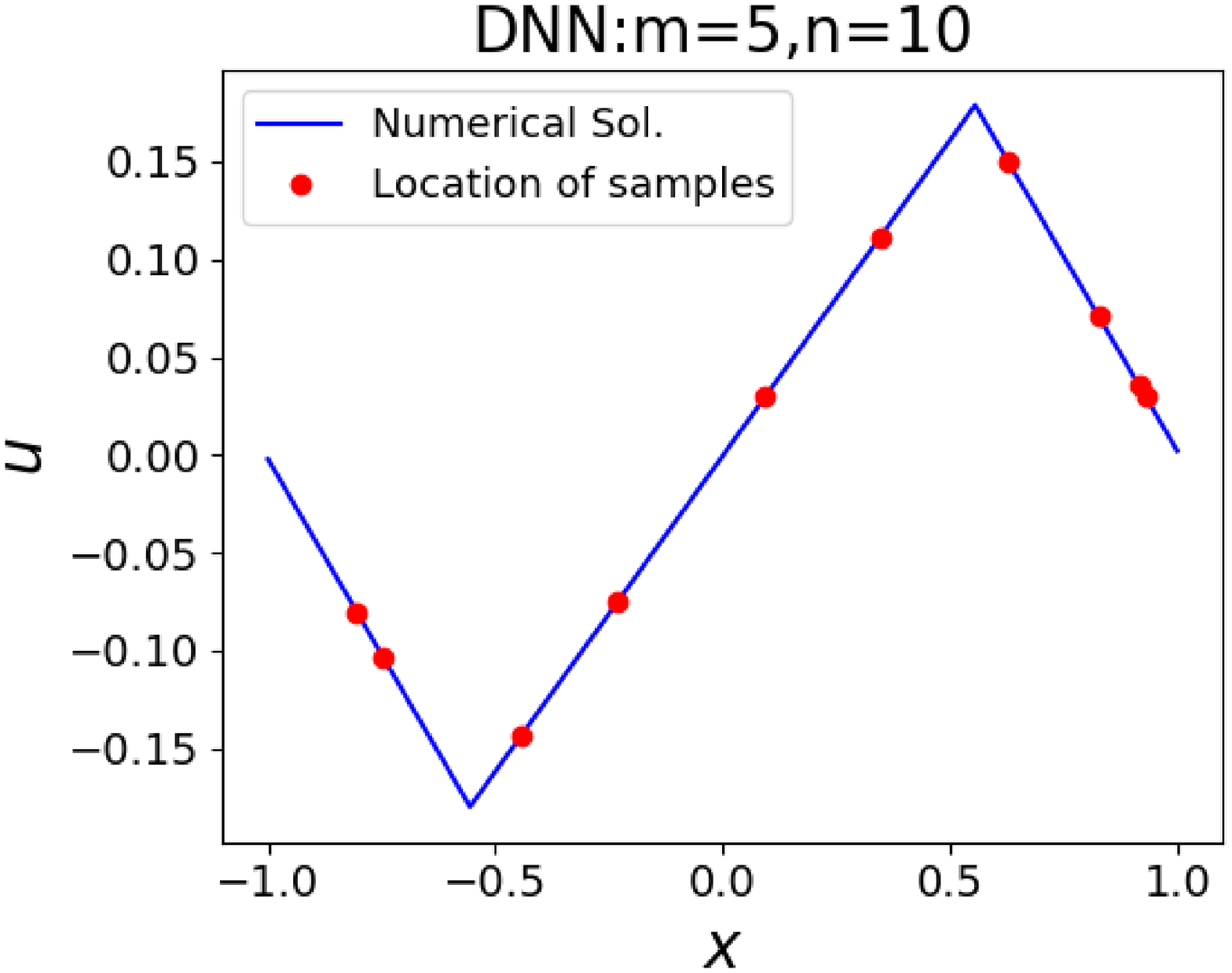} \vspace{0.1in} 
	\includegraphics[width=0.45\textwidth]{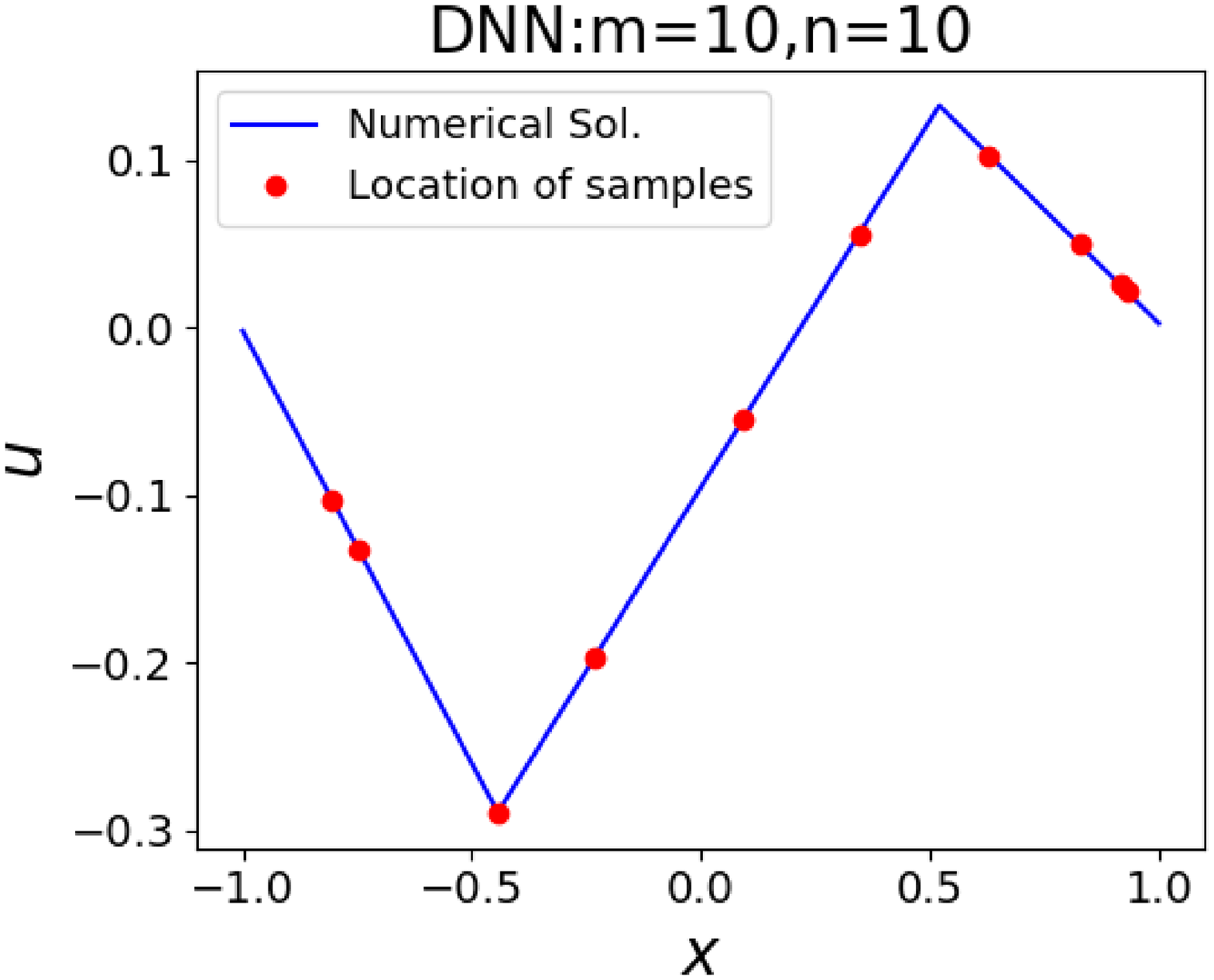} \vspace{0.1in}\\ 
	\includegraphics[width=0.45\textwidth]{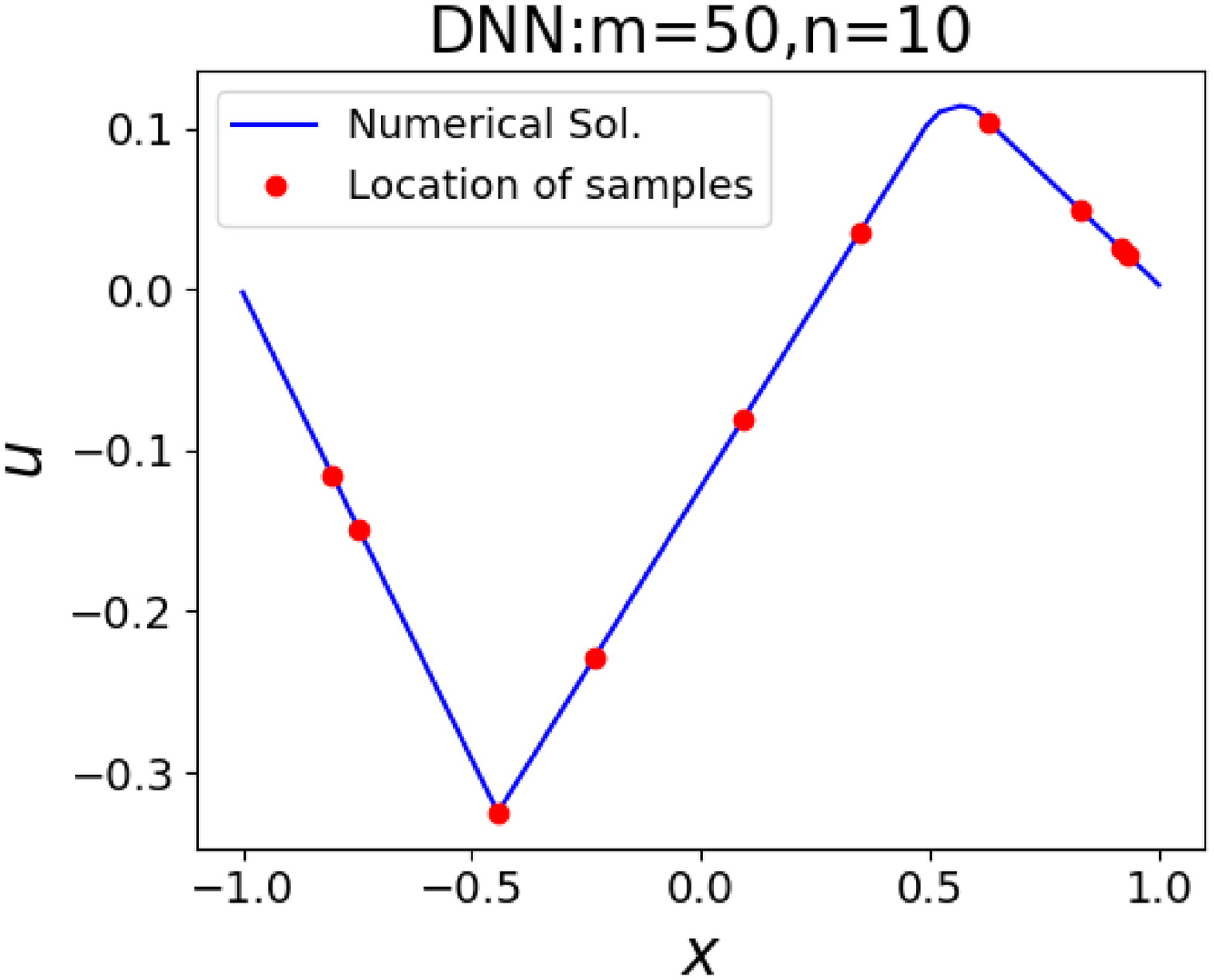} \vspace{0.1in} 
	\includegraphics[width=0.45\textwidth]{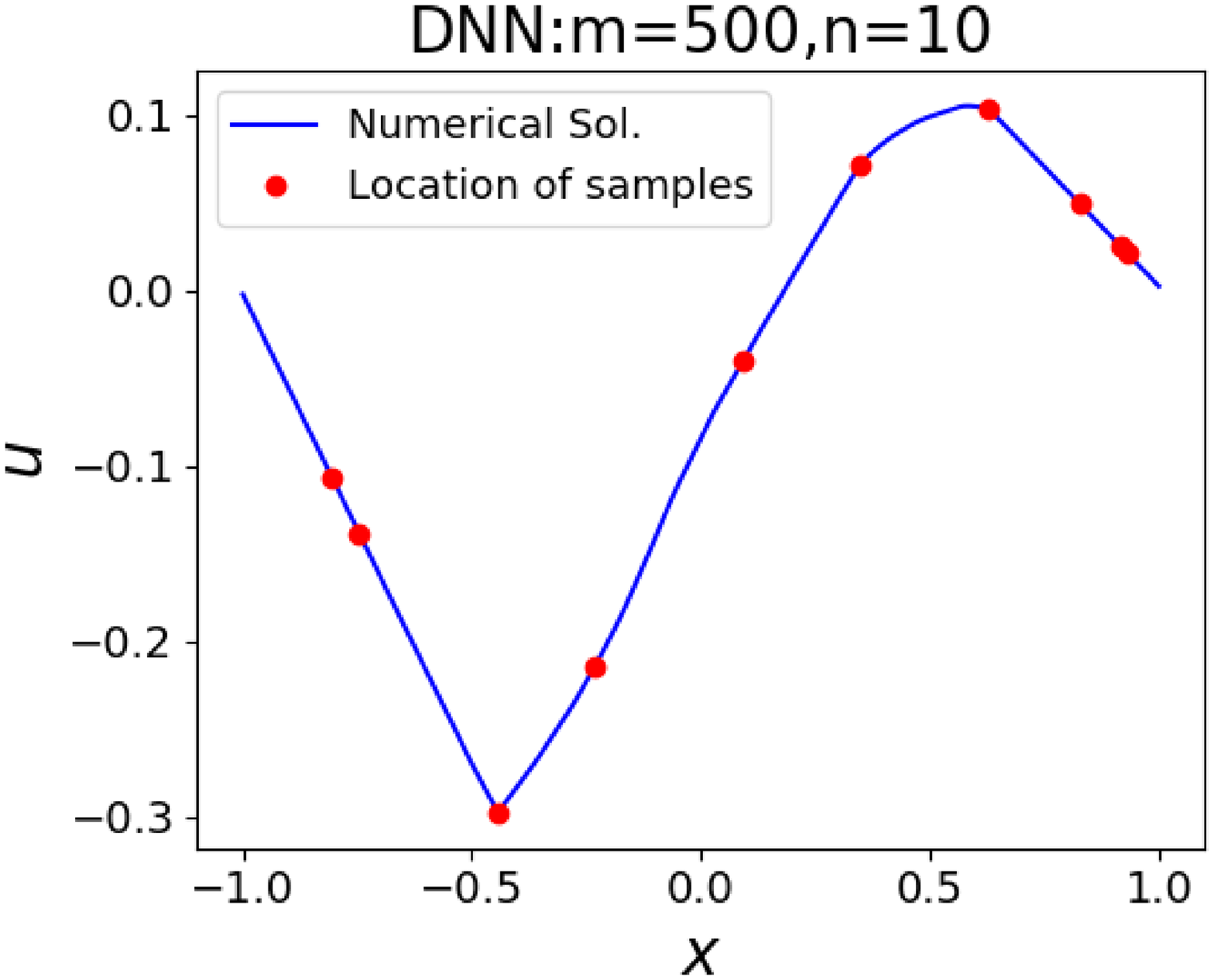} \vspace{0.1in}
\end{minipage}
\caption{(Example 2): Numerical solutions in DNN method with ReLU activation function and 10 sampling points.}
\label{DNNrelu}
\end{figure}

\bigskip
\noindent{\bf Example 3:} We  consider the 2D case 
\begin{eqnarray*}
\left\{\begin{array}{c}
 - \Delta u(\bm{x})=f(\bm{x}),\quad \bm{x}\in(0,1)^2,\\
 u(\bm{x})=0, \quad \bm{x} \in \partial(0,1)^2,
\end{array}
\right.
\end{eqnarray*}
where $\vx=(x, y)$ and we know the values of $f$ at $n$ points sampled from the function $f(\vx):=f(x, y) = 2\pi^2\sin(\pi x)\sin(\pi y).$
We fix the number of sample points $n=5^2$. 
%Set $x_h=[0.1,0.25,0.5,0.8,0.9]$, the position of the sampling points is $x_h\times x_h$. 
The sampling points in the $x$ direction and the $y$ direction are both at $x_h=[0.1,0.25,0.5,0.8,0.9]$. % namely $\vx_i$ 
We test the solution with the number of basis $m=5,50,100,200$, respectively. Fig. \ref{2d} plots the R-G solutions with Legendre basis and piecewise linear basis function. It can be seen that the numerical solution is a function with strong singularity.  Fig. \ref{2d1} shows the profile of R-G solutions at $y=0.5$ for various $m$, in which we can see that the values of numerical solutions at the sampling points get larger and larger with the increase of $m$. However, Fig. \ref{2dDNN} shows that the DNN solutions are stable without singularity for large $m$. 
 %We use the Legendre basis in SM and. 
%Due to the limited computing power, the error between the R-G solution and the exact solution is relatively large. 
%It can be seen that the series expression of the reference solution converges, but the convergence speed is slow, and it is difficult for us to draw an accurate image of the reference solution.

 \begin{figure}[ht!]
\centering
\begin{minipage}[]{\textwidth}
	\includegraphics[width=0.23\textwidth]{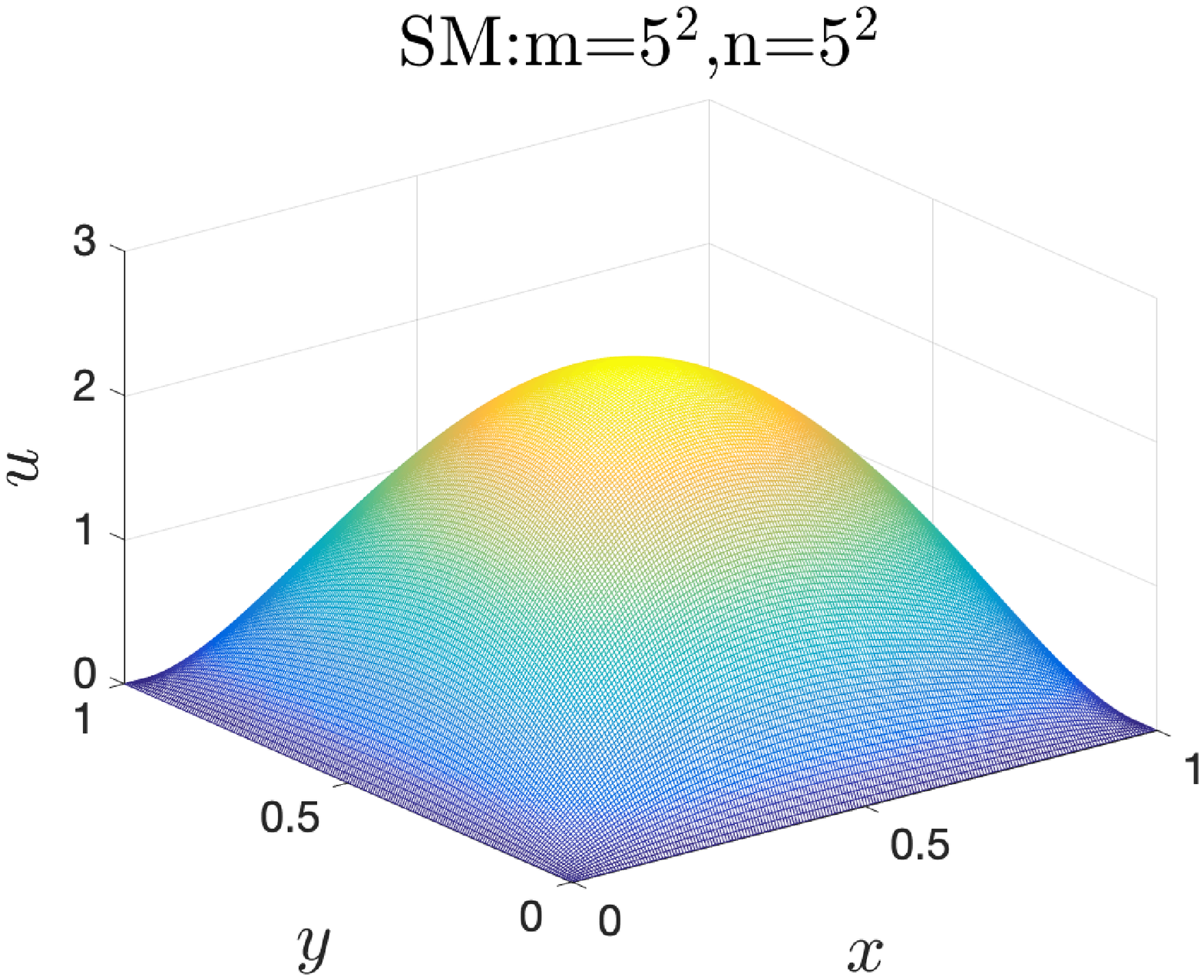}\hspace{0.1in} 
	\includegraphics[width=0.23\textwidth]{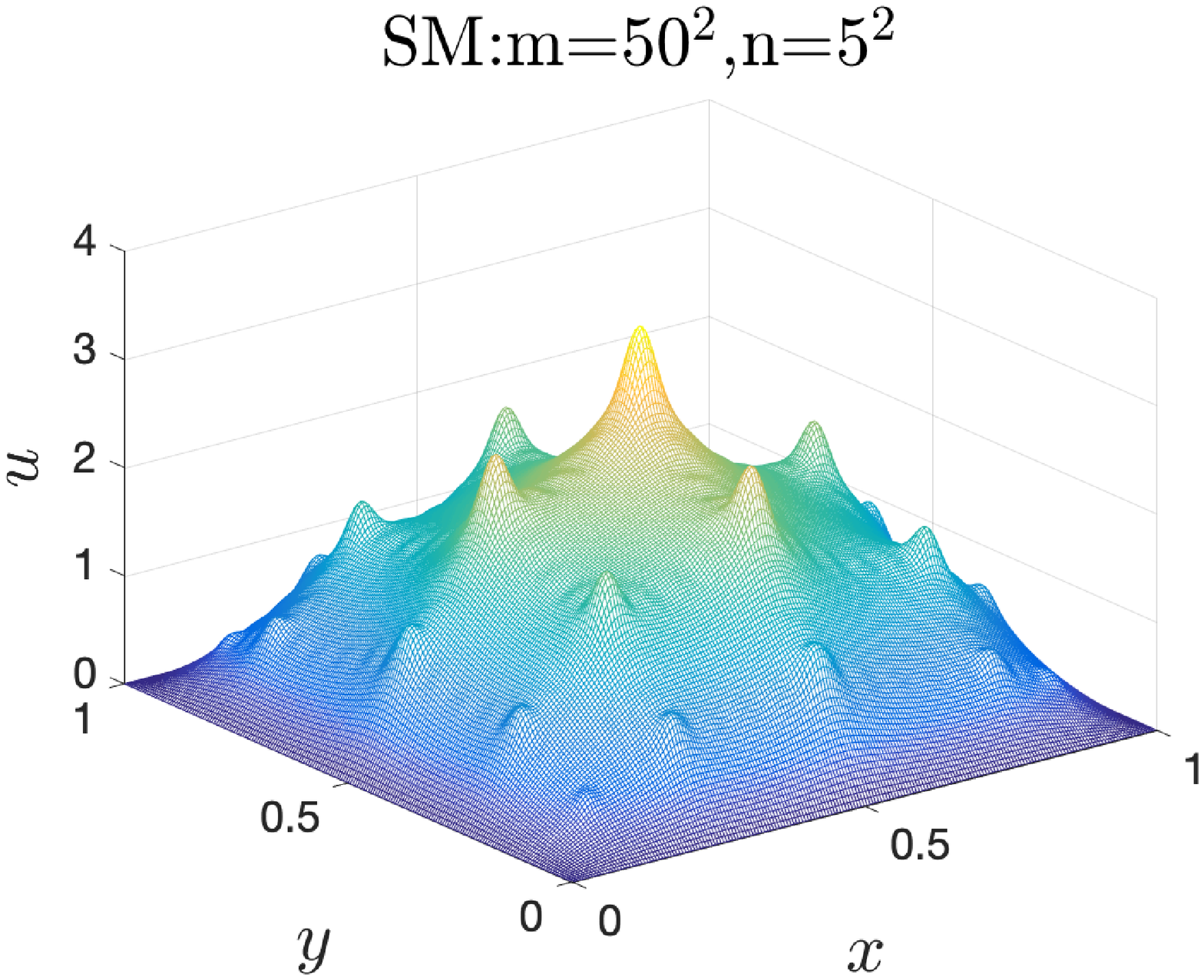}\hspace{0.1in}
	\includegraphics[width=0.23\textwidth]{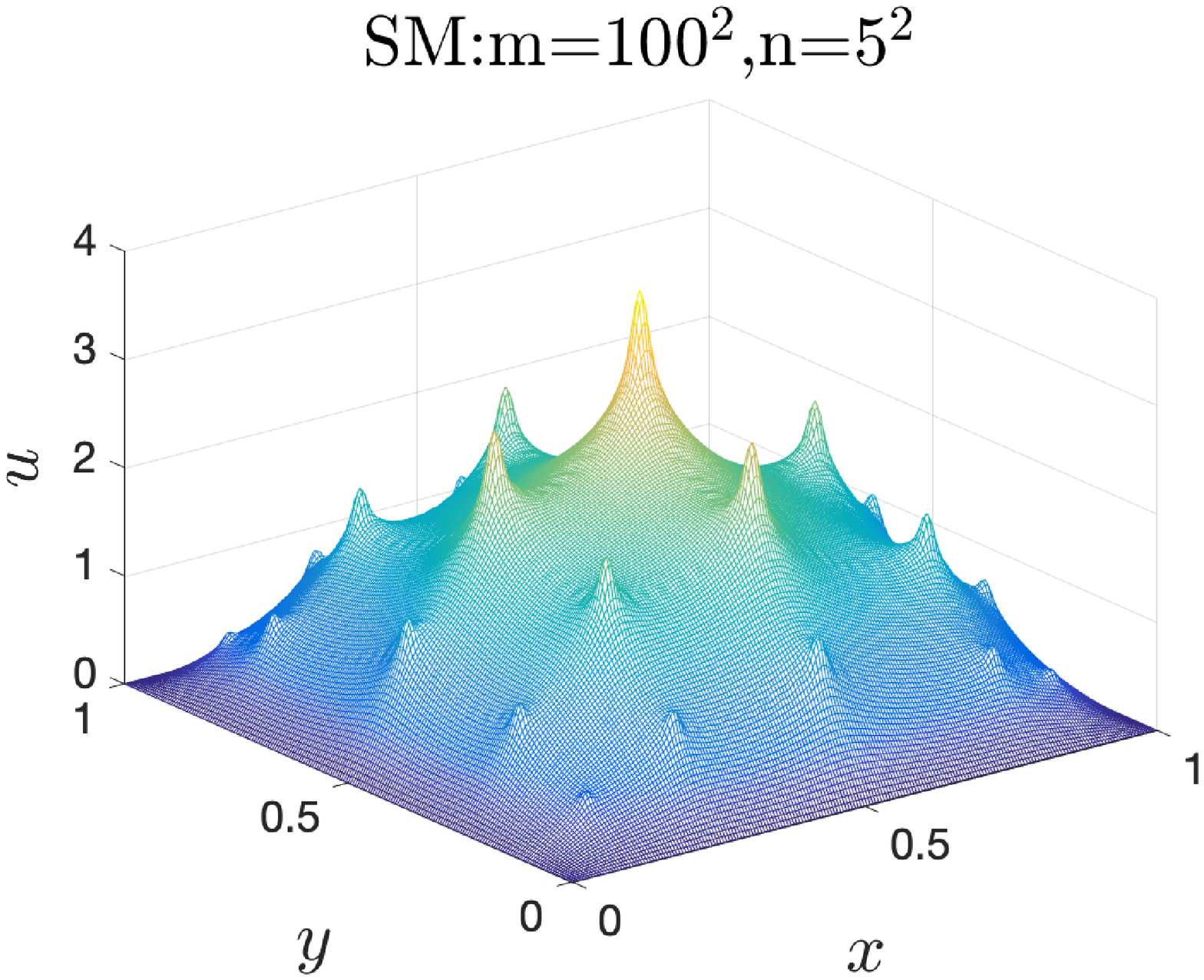}\hspace{0.1in} 
	\includegraphics[width=0.23 \textwidth]{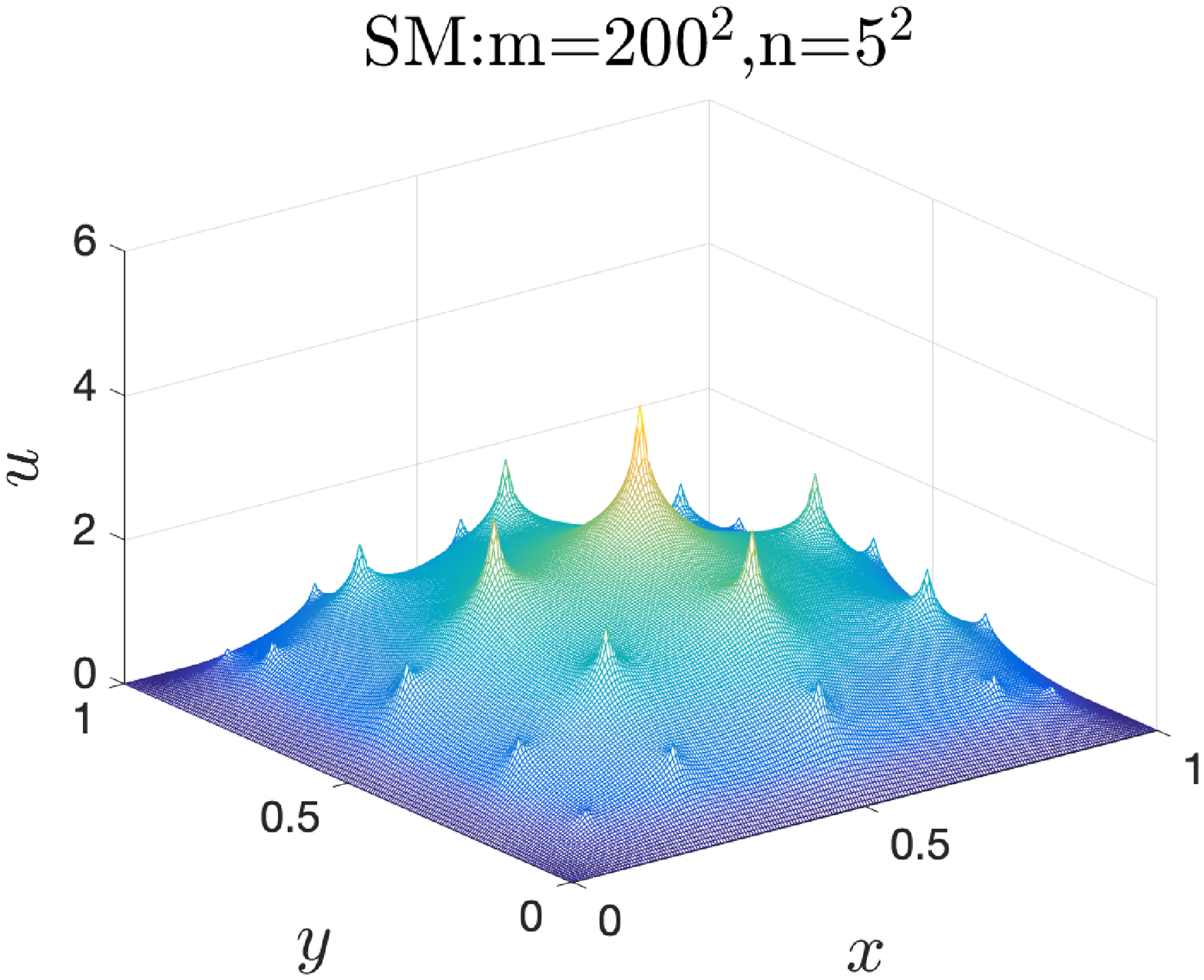}\vspace{0.2in} \\
	\includegraphics[width=0.23\textwidth]{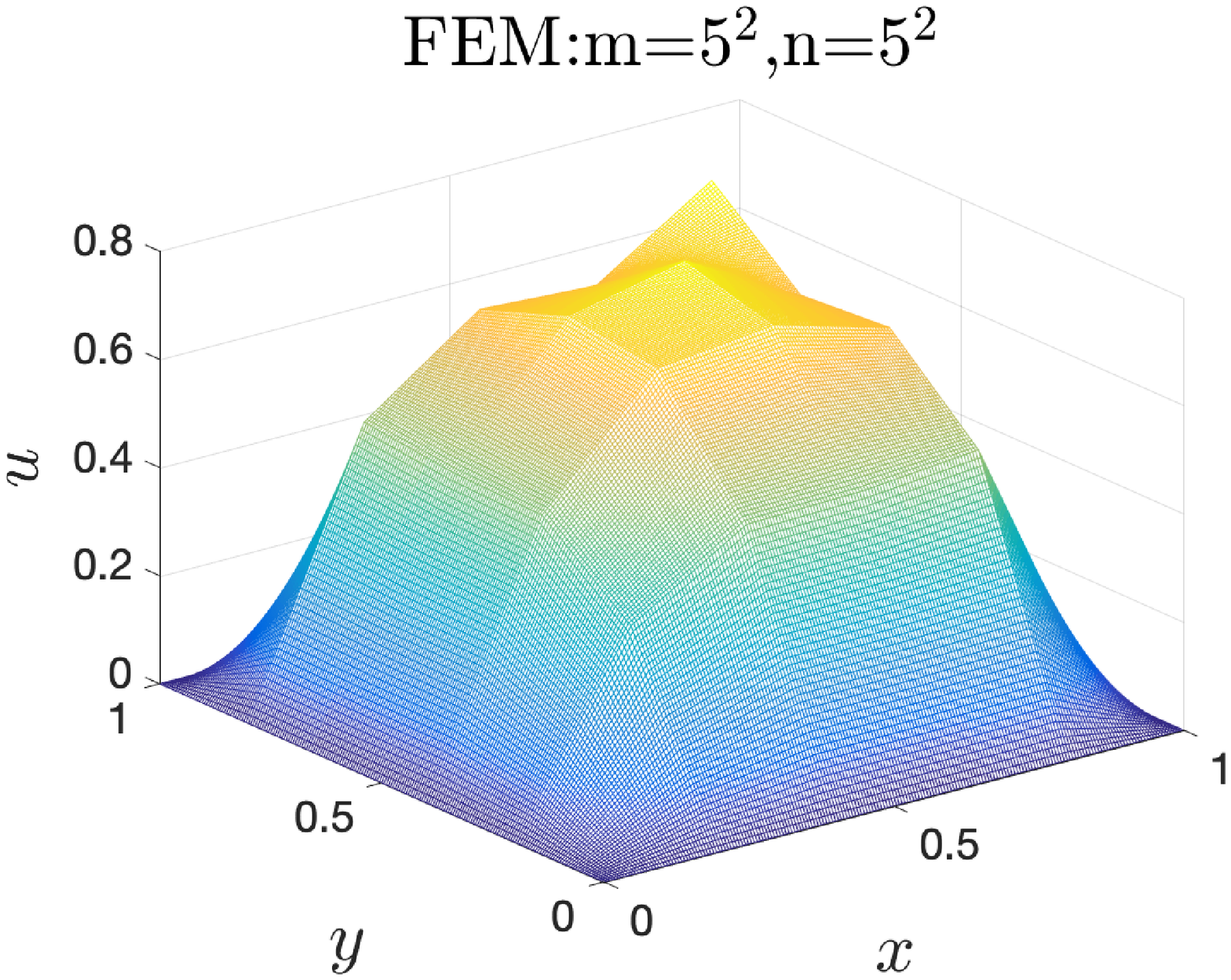}\hspace{0.1in} 
	\includegraphics[width=0.23\textwidth]{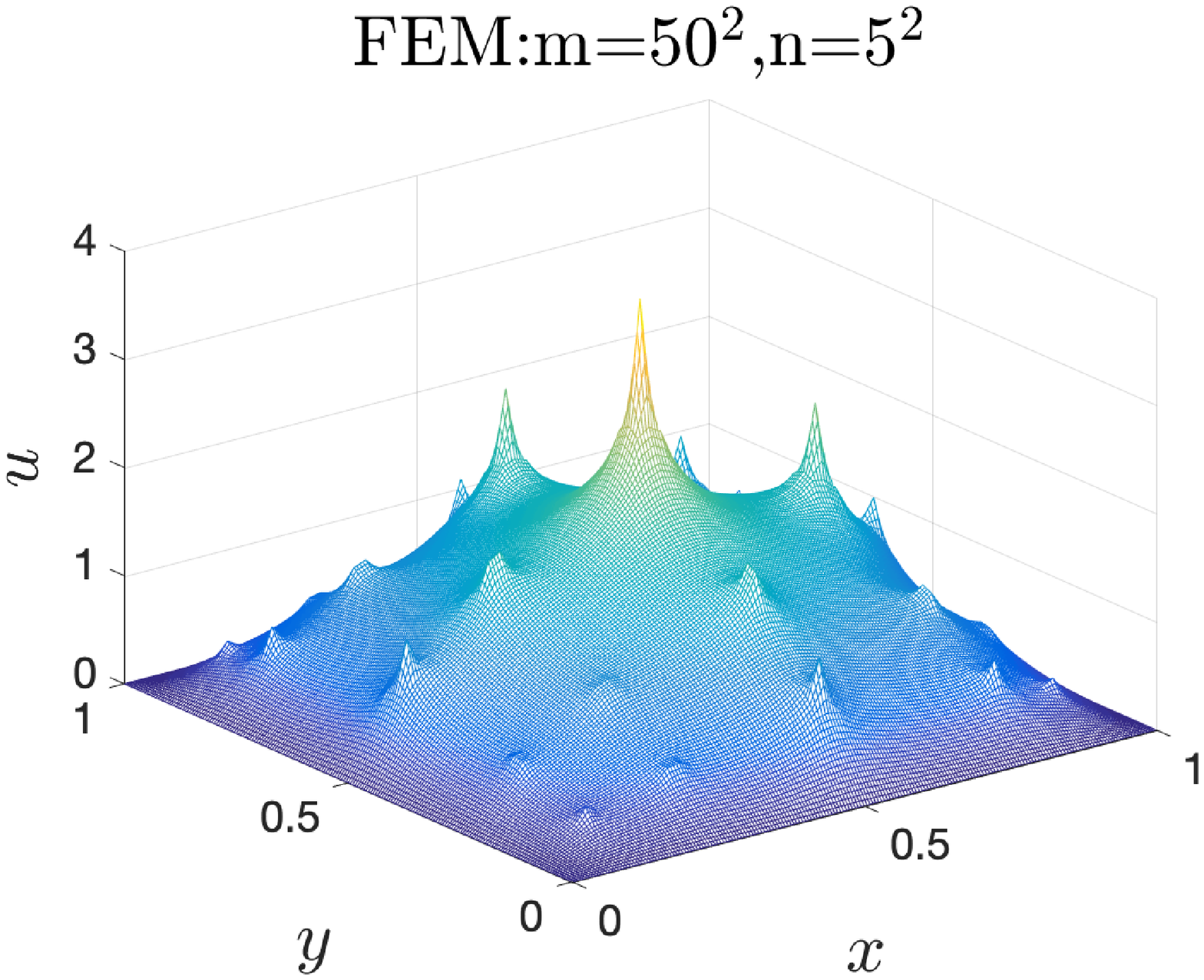}\hspace{0.1in}
	\includegraphics[width=0.23\textwidth]{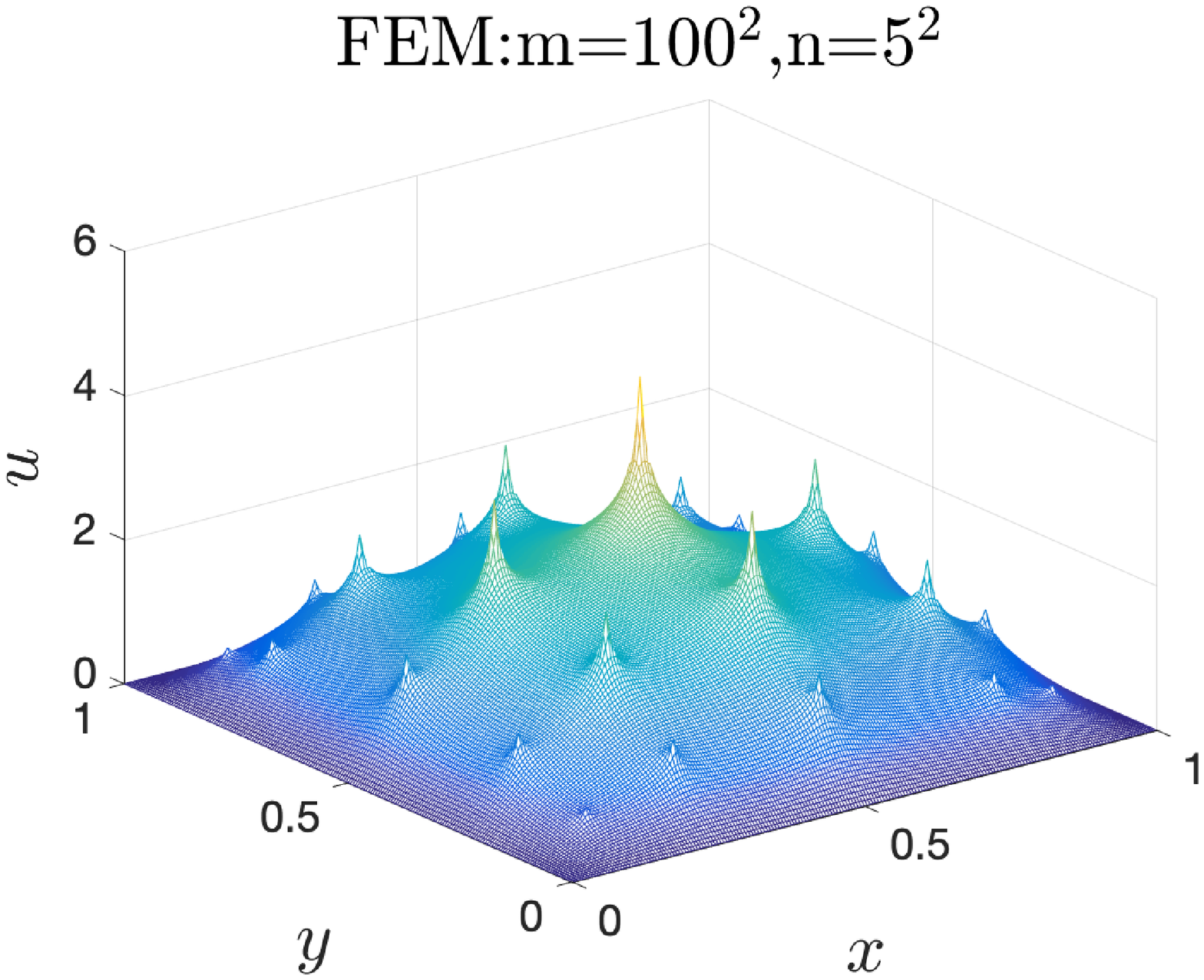}\hspace{0.1in} 
	\includegraphics[width=0.23 \textwidth]{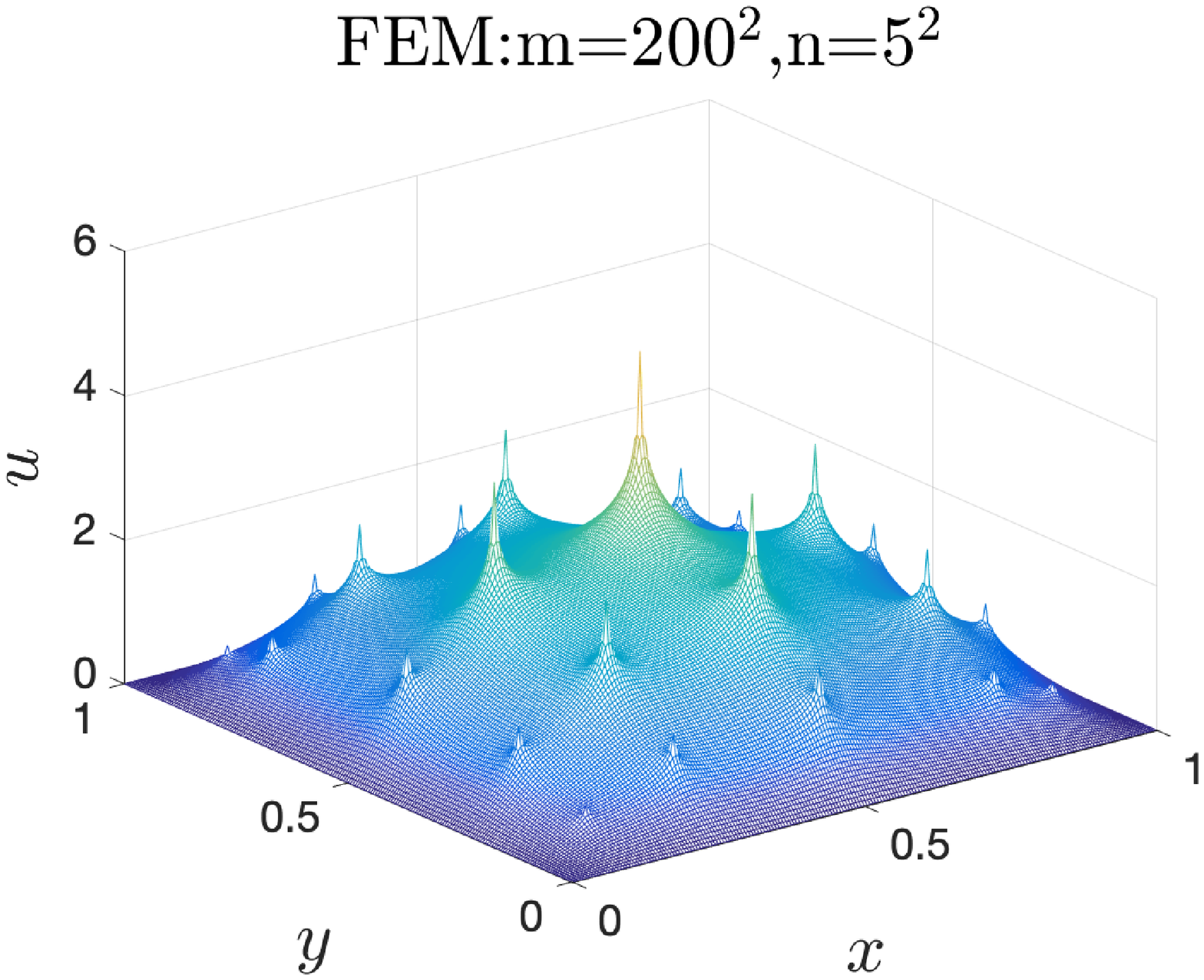}\\
\end{minipage}
\caption{\label{2d} (Example 3): R-G solutions with different $m$. The basis functions for the first and the second row are Legendre basis function and piecewise  linear  basis  function, respectively.}
\end{figure}

 \begin{figure}[ht!]
\centering
\begin{minipage}[]{\textwidth}
	\includegraphics[width=0.45\textwidth]{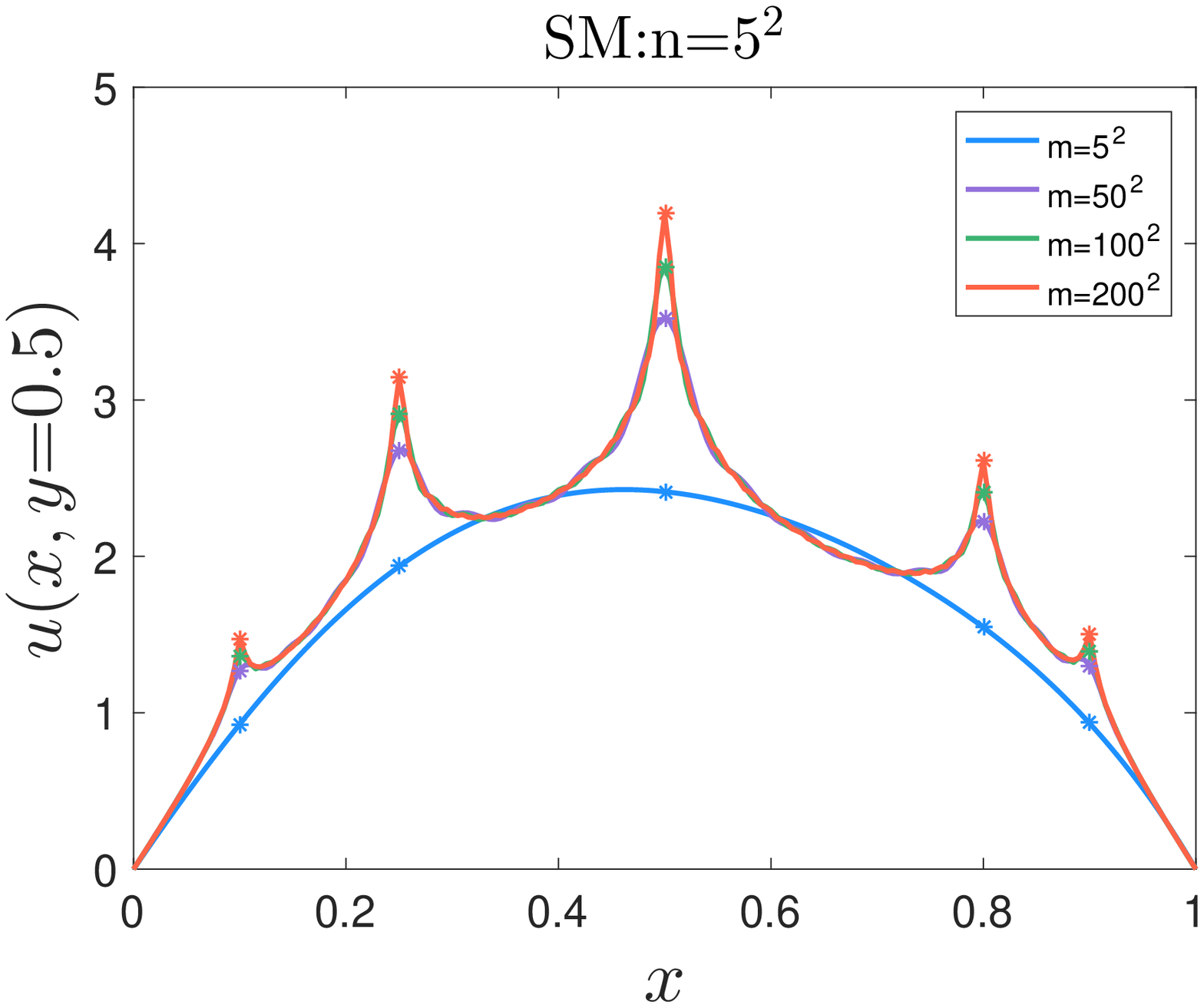}\hspace{0.2in} 
	\includegraphics[width=0.45\textwidth]{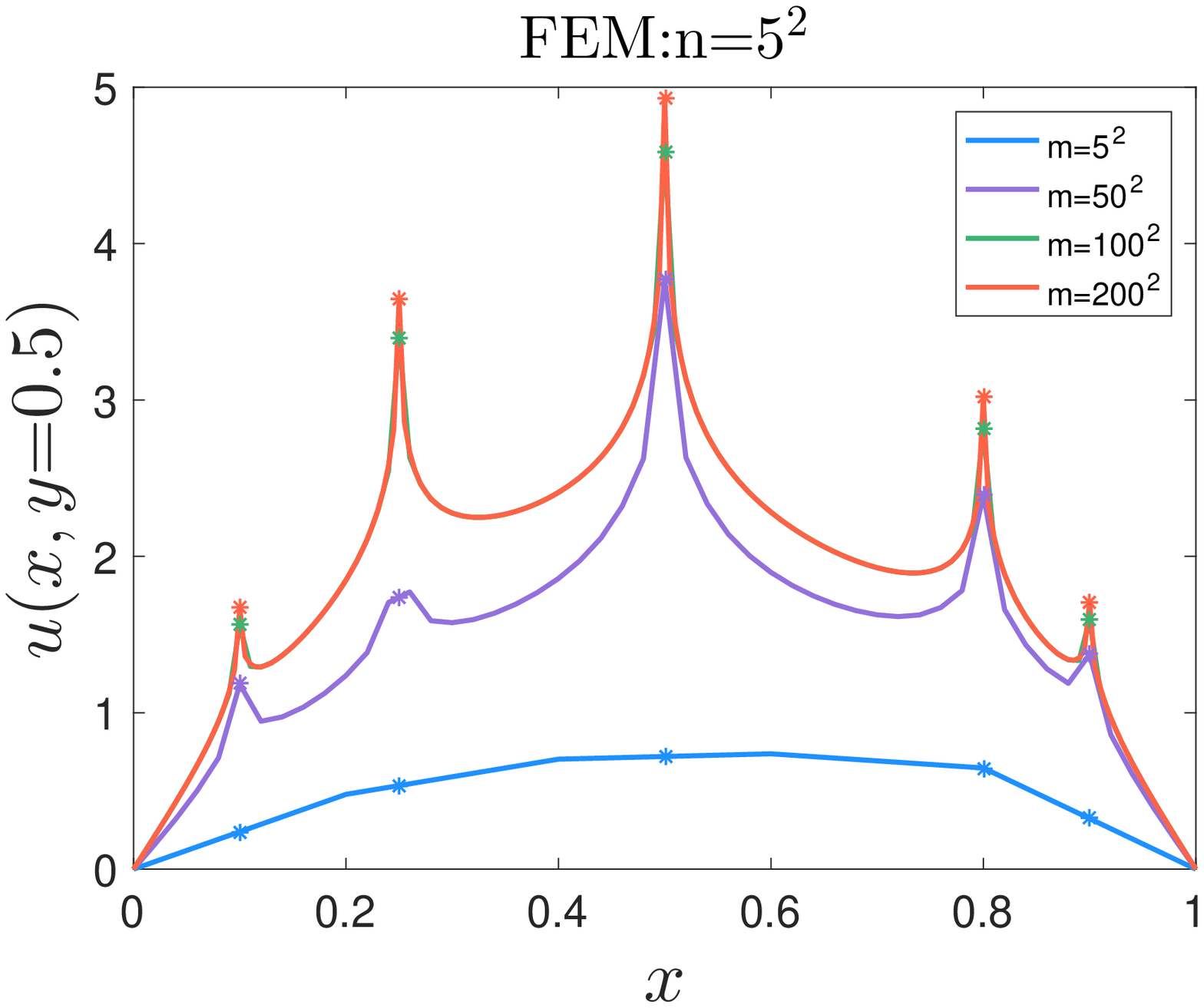}
\end{minipage}
\caption{\label{2d1} (Example 3): Profile of R-G solutions with different $m$.}
\end{figure}

 \begin{figure}[ht!]
\centering
\begin{minipage}[]{\textwidth}
	\includegraphics[width=0.24\textwidth]{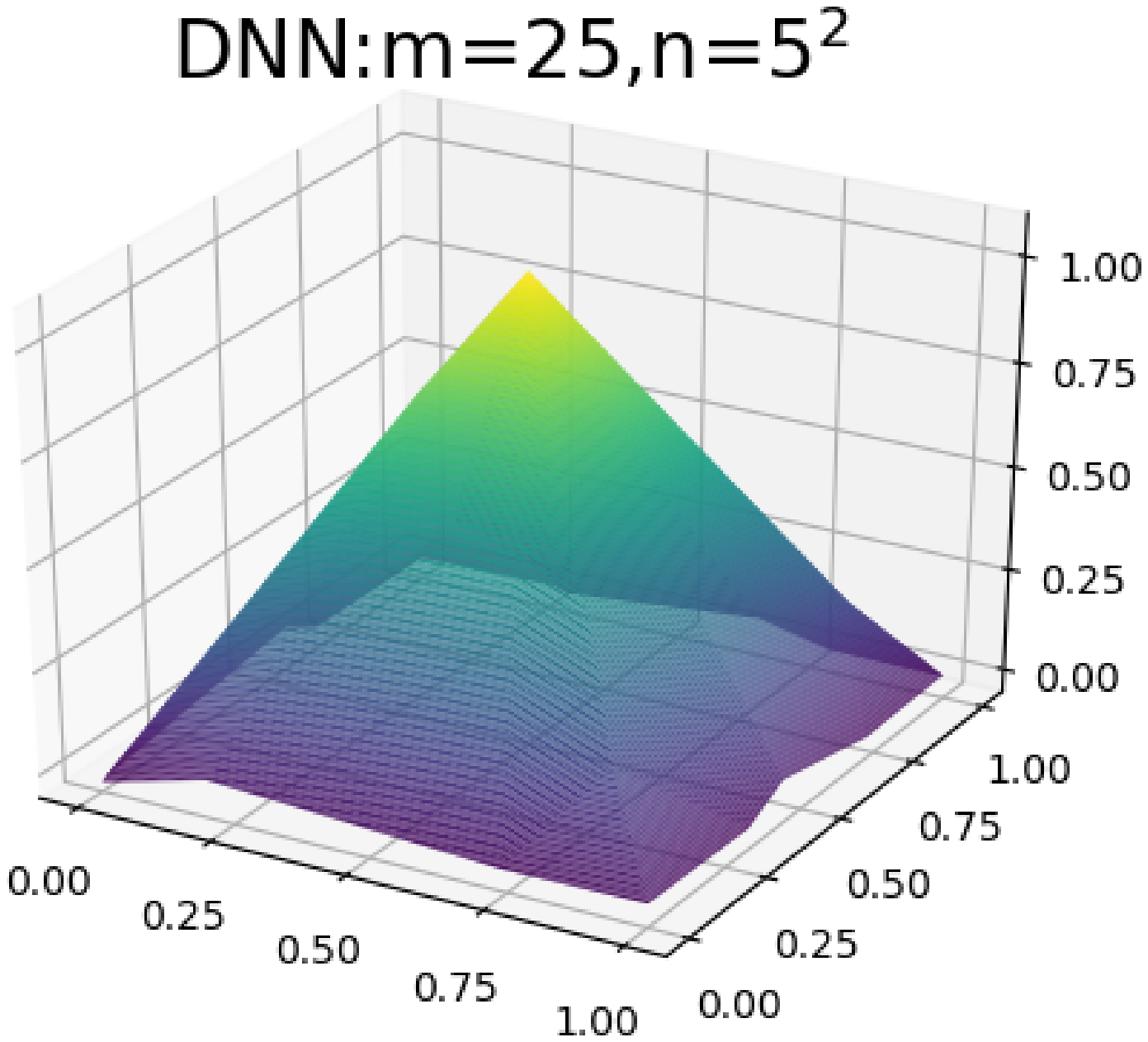}
	\includegraphics[width=0.24\textwidth]{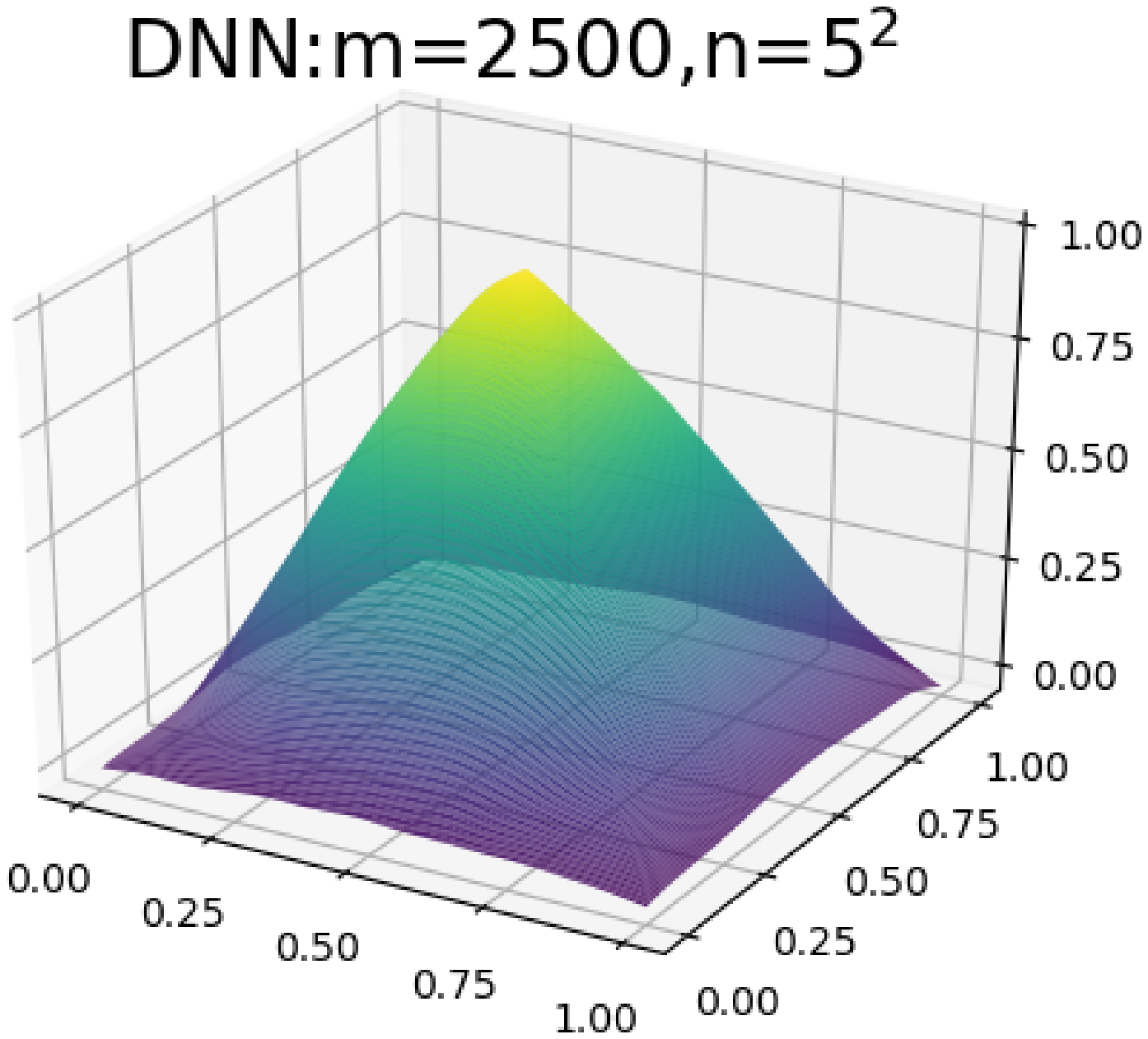}
	\includegraphics[width=0.24\textwidth]{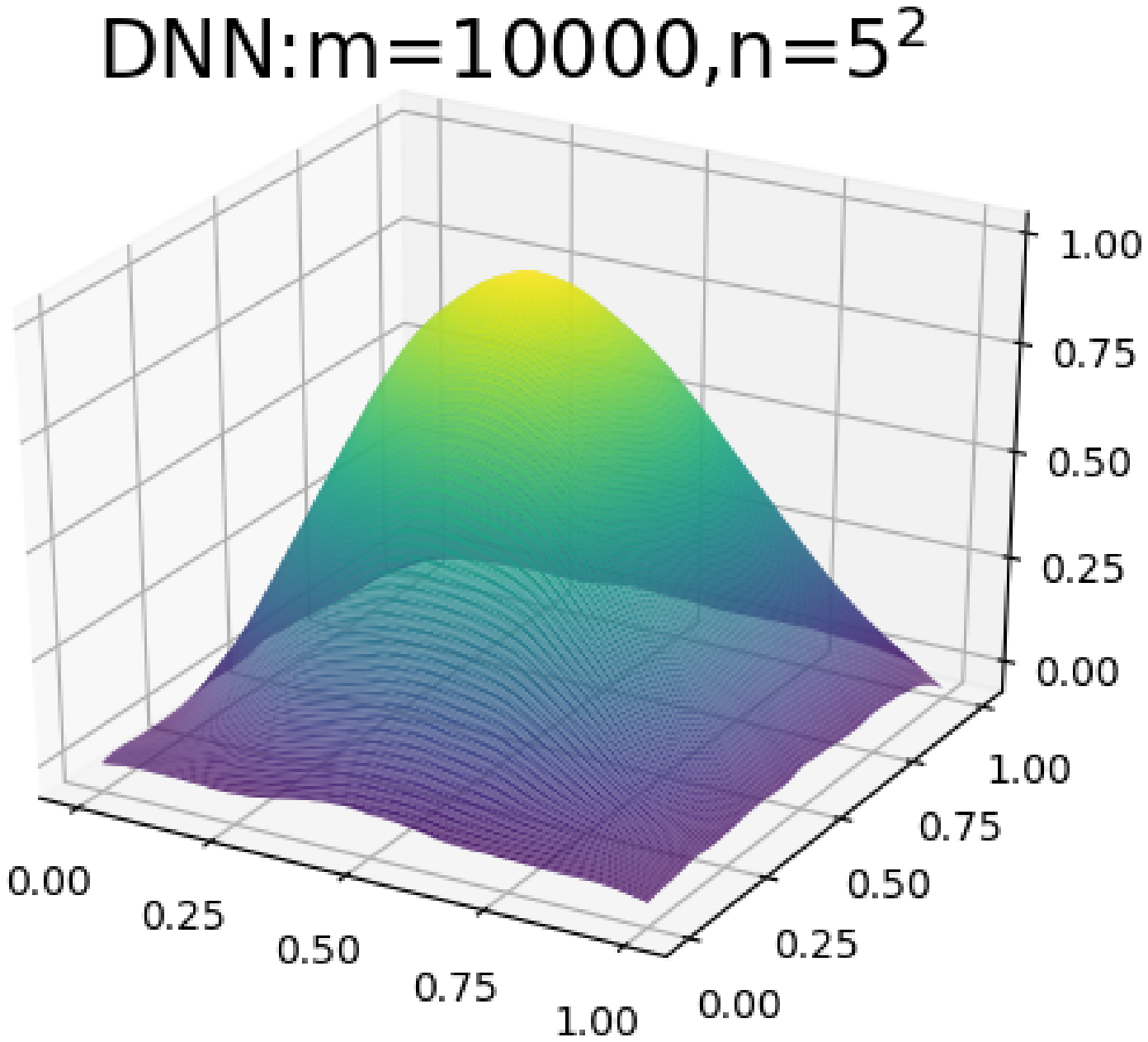}
	\includegraphics[width=0.24 \textwidth]{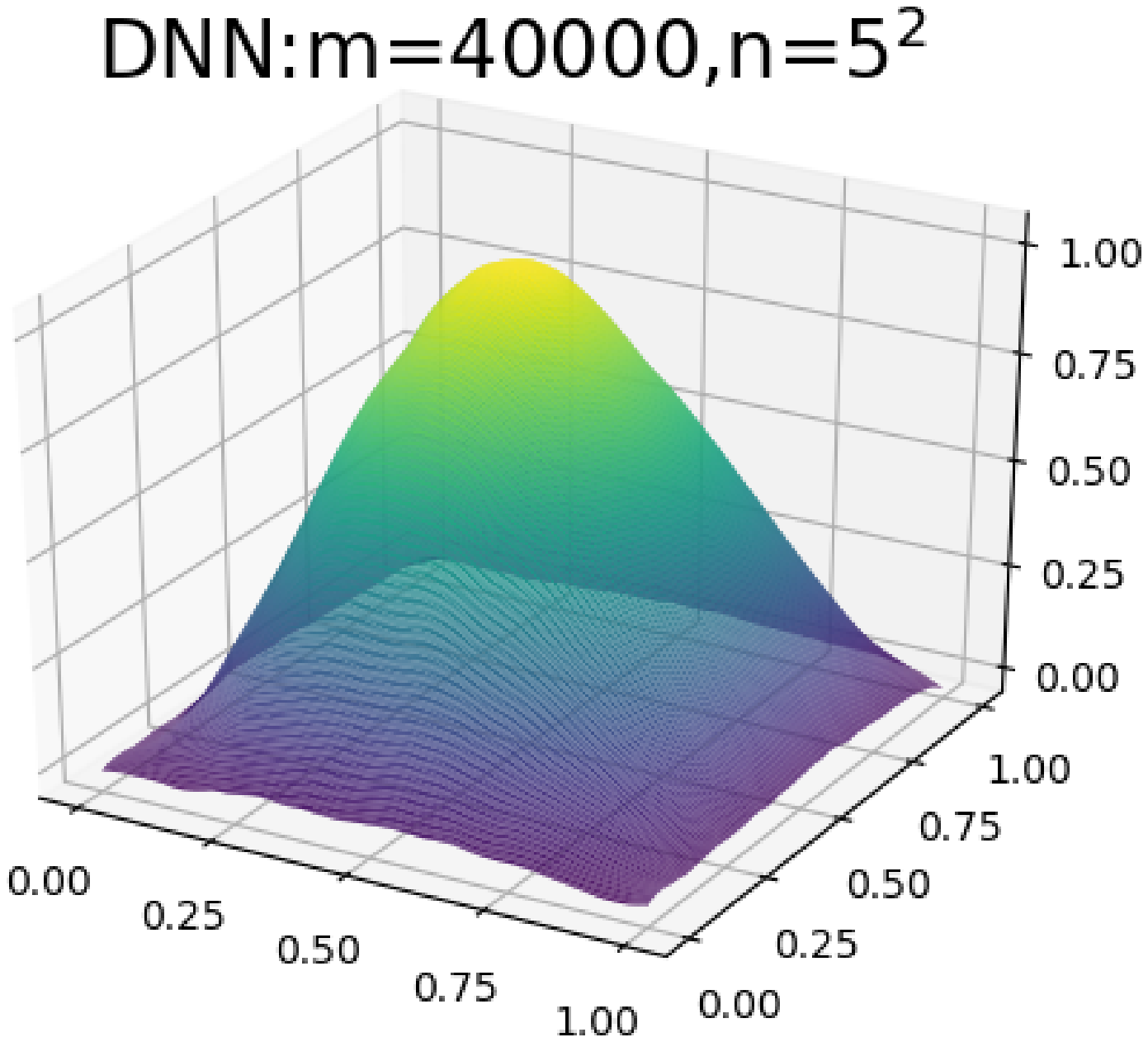}\vspace{0.2in} \\
	\includegraphics[width=0.24\textwidth]{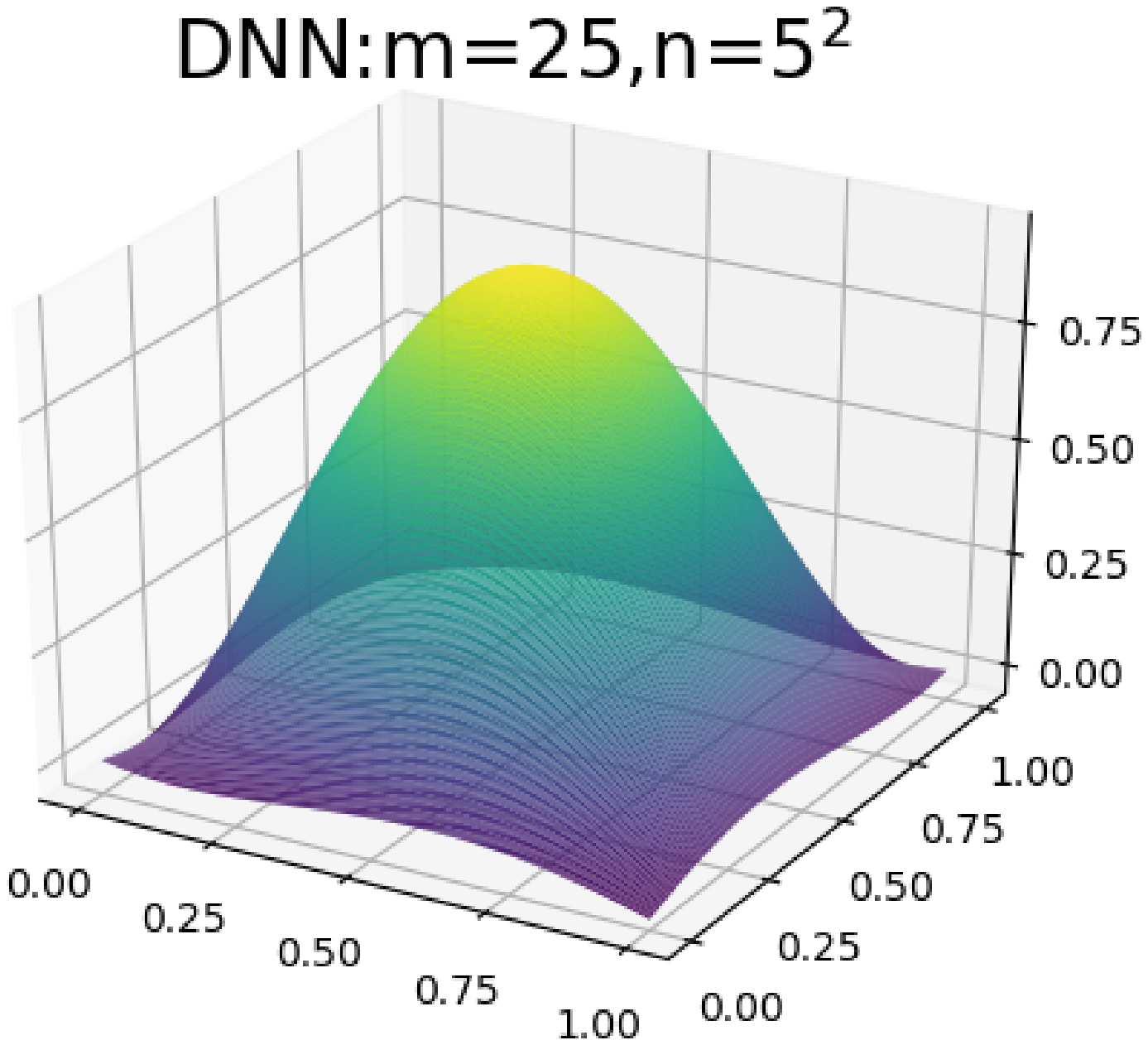}
	\includegraphics[width=0.24\textwidth]{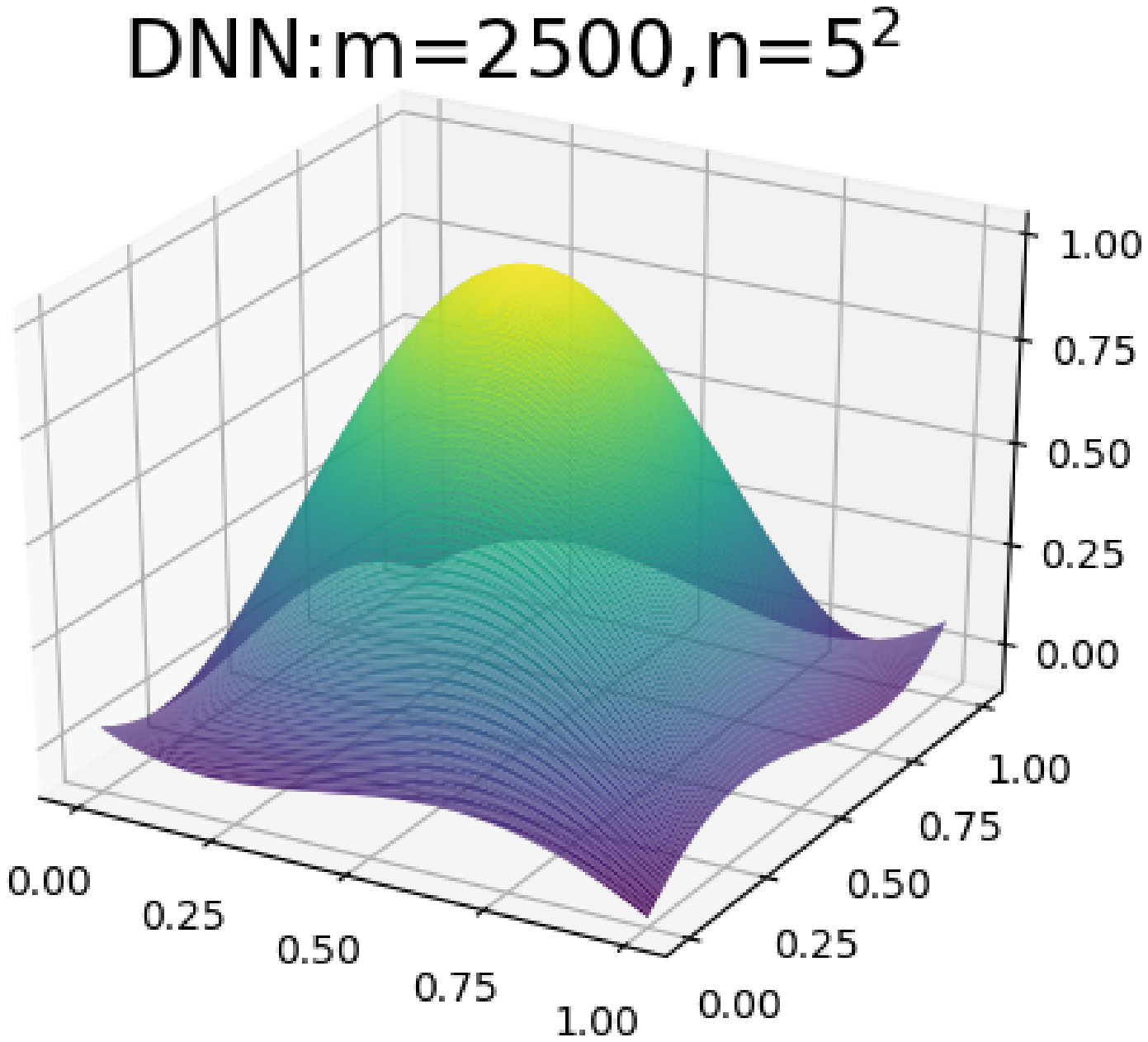}
	\includegraphics[width=0.24\textwidth]{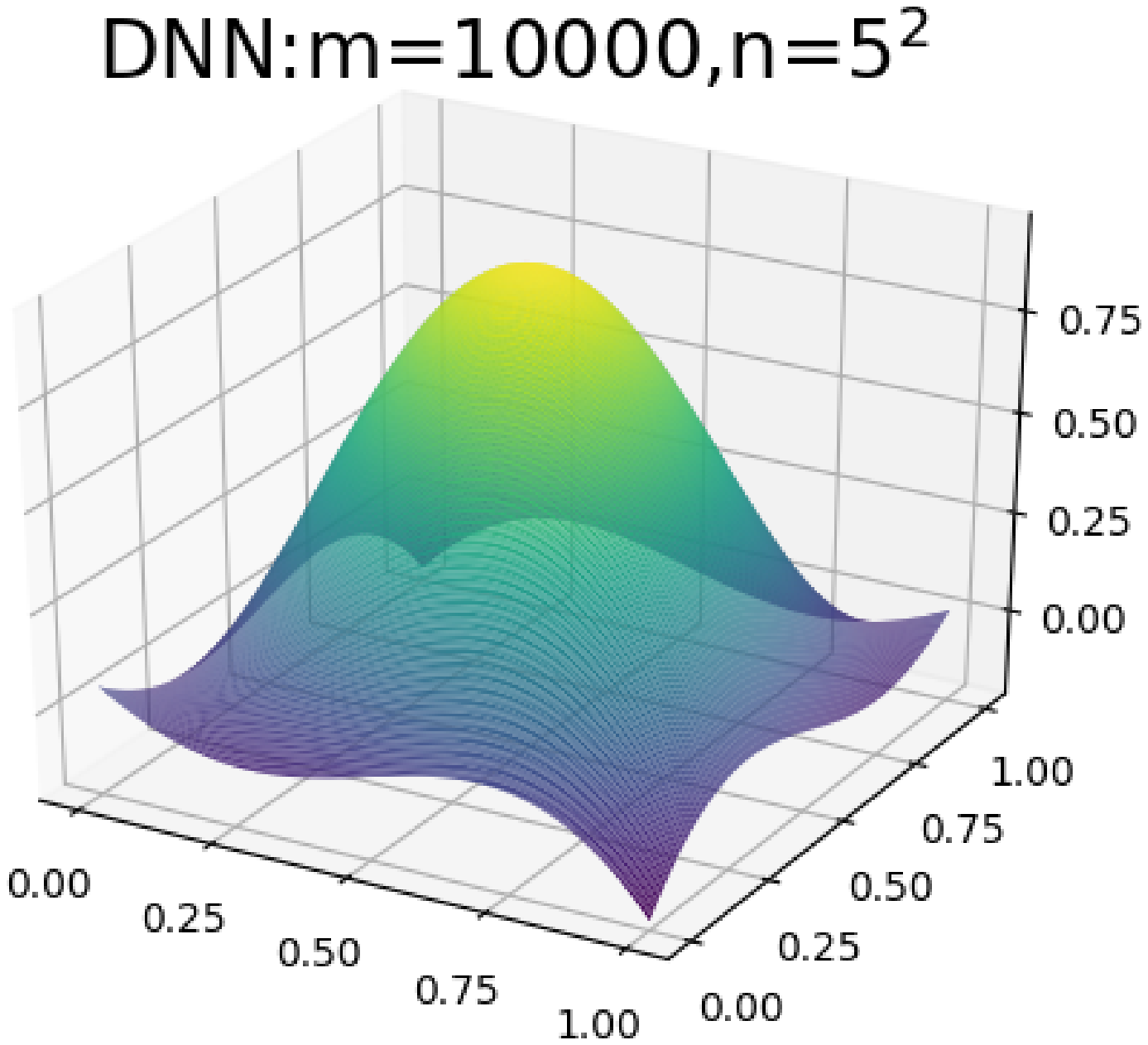}
	\includegraphics[width=0.24 \textwidth]{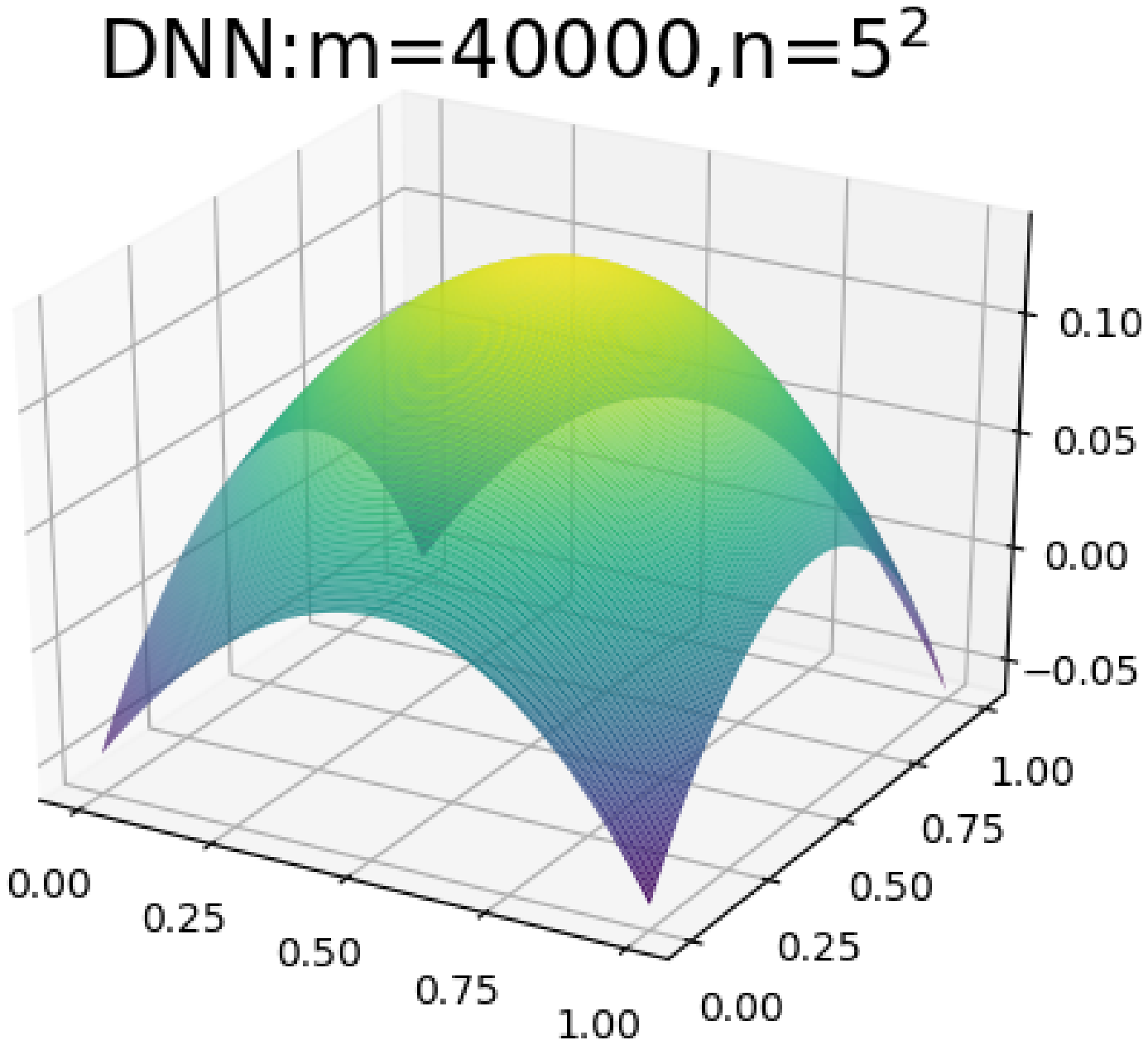}\\
\end{minipage}
\caption{\label{2dDNN} (Example 3): DNN solutions with different $m$. The activation functions for the first and the second row are ${\rm ReLU} (x)$ and $\sin (x)$, respectively.}
\end{figure}

\section{Discussion of DNN method in solving PDE}

DNNs are widely used in solving PDEs, especially for high-dimensional problems. The optimizing the loss functions in Eqs. \eqref{eq:lossdirect} and \eqref{eq:Energy-1} are equivalent to solving (\ref{RG}) except that the bases in \eqref{eq:lossdirect} and \eqref{eq:Energy-1} are adaptive. In addition, the DNN problem is optimized by (stochastic) gradient descent. The experiments in the previous section have shown that when the number of bases goes to infinity, DNN methods solve (\ref{Model}) by a relatively smooth and stable function compared with the one obtained by Theorem \ref{thm1}. We now utilize the F-Principle to understand what leads to the smoothness.  

For two-layer wide DNNs with $d=1$, the two terms of $\Gamma(\vxi)$ in the minimization problem of (\ref{eq: minFPnorm}) yield different fitting results. Note that $\min \int_{R} ||\vxi||^{-2}|\fF[h](\vxi)|^{2}\D\vxi$ leads to a piecewise linear function, while $\min \int_{R} ||\vxi||^{-4}|\fF[h](\vxi)|^{2}\D\vxi$ leads to a cubic spline. Since the DNN is a combination of both terms, therefore, the DNN would yields to a much smoother function than the piecewise linear function. For a general DNN, the coefficient  $\Gamma(\vxi)$  in (\ref{gdform}) cannot be obtained exactly, however,  the monotonically decreasing property of  $\Gamma(\vxi)$ with respect to $\vxi$ can be postulated based on the F-Principle.

\section{Conclusion}
This paper compares the different behaviors of Ritz-Galerkin method and DNN method through solving PDEs to better understand the working principle of DNNs. We consider a particular Poisson problem \eqref{Model}, where the r.h.s. term $f$ is a discrete function. We analyze why the two numerical methods behave differently in theory. R-G method deals with the discrete $f$ as the linear combination of Dirac delta functions, while DNN methods implicitly bias towards functions with more low-frequency components to interpolate the discrete sampling points due to the F-Principle. Furthermore, from the numerical experiments, as the number of bases increases, one can see that the solutions obtained by R-G method approximate piecewise linear functions for 1d case and singular function for the 2d case, regardless of the basis function, but the solutions obtained by DNN method are smoother for 1d case and stable for the 2d case. In conclusion, based on the theoretical and numerical study in comparison to traditional methods in solving PDEs, DNN method possesses special implicit bias towards low frequencies, which leads to a well-bahaved solution even in a heavily over parameterized setting.

%%%% Acknowledgments %%%%%%%%
\section*{Acknowledgments}
Zhiqin Xu is supported by National Key R\&D Program of China (2019YFA0709503), Shanghai Sailing Program, Natural Science Foundation of Shanghai (20ZR1429000), and partially supported by HPC of School of Mathematical Sciences at Shanghai Jiao Tong University. Jihong Wang and Jiwei Zhang is partially supported by NSFC under No. 11771035 and NSAF U1930402, and the Natural Science Foundation of Hubei Province No. 2019CFA007, and Xiangtan University 2018ICIP01.

%%%% Bibliography  %%%%%%%%%%
%%\bibliographystyle{spphys}       % APS-like style for physics
%\bibliography{DLRef} 

\end{document}